\documentclass[pdflatex,sn-mathphys-num]{sn-jnl}

\usepackage{enumerate}
\usepackage{bbm}
\usepackage{algpseudocode}

\usepackage{graphicx}%
\usepackage{multirow}%
\usepackage{amsmath,amssymb,amsfonts}%
\usepackage{amsthm}%
\usepackage{mathrsfs}%
\usepackage[title]{appendix}%
\usepackage{xcolor}%
\usepackage{textcomp}%
\usepackage{manyfoot}%
\usepackage{booktabs}%
\usepackage{algorithm}%
\usepackage{algorithmicx}%
\usepackage{algpseudocode}%
\usepackage{listings}%

\theoremstyle{thmstyleone}%
\newtheorem{theorem}{Theorem}
\newtheorem{lemma}[theorem]{Lemma}%
\newtheorem{corollary}[theorem]{Corollary}%
\newtheorem{remark}[theorem]{Remark}%

\theoremstyle{thmstyletwo}%

\theoremstyle{thmstylethree}%

\newcommand{\N}{\mathbb{N}}
\newcommand{\R}{\mathbb{R}}

\newcommand{\LL}{\mathcal{L}}

\newcommand{\sS}{\mathbb{S}}
\newcommand{\1}{\ensuremath{\mathbbm{1}}}

\newcommand{\kernel}{\mathcal N}

\newcommand{\sub}[1]{{\ensuremath\scriptscriptstyle\operatorname{#1}}}  
\newcommand{\Adm}{\mathcal Y_\text{adm}}
\newcommand{\dx}[1][x]{\ensuremath{\,{\rm{d}} #1}}

\def\norm#1{\hspace{0.2ex} \|#1\| \hspace{0.2ex}} 
 
\newcommand{\kommentar}[1]{}

\DeclareMathOperator{\trace}{tr} 

\begin{document}

\title[A fully globalized solver]{A fully globalized solver for discretized inverse elliptic coefficient problems with exact data}

\author*[1]{\fnm{Bastian} \sur{Harrach}}\email{harrach@math.uni-frankfurt.de}

\affil*[1]{\orgdiv{Institute for Mathematics}, \orgname{Goethe-University Frankfurt}, \orgaddress{\street{Robert-Mayer-Str. 10}, \city{Frankfurt}, \postcode{60325}, \country{Germany}}}

\abstract{
We consider finite-dimensional nonlinear inverse problems arising from finite element discretizations of elliptic inverse coefficient problems such as the Calder\'on problem with finitely many measurements and unknowns. Such inverse coefficient problems are notorious for their nonlinearity and ill-posedness, and numerical solvers tend to depend strongly on good initial values. In this work, we develop a new locally convergent algorithm with an explicit residual criterion that ensures convergence to the inverse problem solution, and a globalized variant that is guaranteed to automatically switch to the faster locally convergent algorithm after finitely many global search steps. 
}




\maketitle

\section{Disclaimer} This is a preliminary draft version. It is lacking references, an
introduction and a numerical-results section. 

\section{Setting and first results}\label{subsect:setting}

We consider the finite-dimensional nonlinear inverse problem
\begin{equation}\label{eq:IP}
\text{reconstruct } \quad \hat y\in [y^-,y^+]\subset \R^m \quad \text{ from } \quad \hat Y:=F(\hat y)\in \sS^l_{++},
\end{equation}
where $F:\ \R^m_{++}\to \sS^l$ is given by
\[
F(y)=B  A_y^{-1} B^T,\quad A_y:=\sum_{j=1}^m y_j A_j.
\]
Here, $l,n,m\in \N$, and $y^+,y^-\in \R^m$ satsify $y^+>y^->0$ componentwise. We assume a known box margin $\delta>0$, i.e.
\[
\hat y\in [y^-+\delta,y^+-\delta \1].
\]
Inequalities between vectors are always understood componentwise, and, for symmetric matrices $\succeq$ denotes the Loewner order. We assume that $0\neq A_i\in \sS^n_+$, $i=1,\ldots,m$, $A_\1\succ 0$, and that $B\in \R^{l\times n}$ has full row rank. Note that $A_\1\succ 0$ implies invertibility of $A_y$ for all $y\in (0,\infty)^m$, so that $F$ is well defined. Moreover, we assume that the inverse problem is uniquely solvable, and Lipschitz stable, i.e.\ that there exists $L>0$ so that
\begin{equation}
\label{eq:Lipschitz_stability}
\norm{y-z}_\infty\leq L \norm{F(y)-F(z)}_2 \quad \text{ for all } y,z\in [y^-,y^+].
\end{equation}
Let us stress that we do not assume knowledge of the Lipschitz constant $L$.

Notably, the discretized Calder\'on problem with finitely many measurements and unknowns, that models Electrical Impedance Tomography, can be written in this form as we demonstrate in subsection \ref{subsect:EIT_FEM}.  In such applications, $m$ is the number of unknowns, the measurements are given by a symmetric $l\times l$-matrix, and $n$ is the number of degrees of freedom of the FEM implementation, so that $n$ is usually very large (and $n\to \infty$ corresponds to solving the underlying PDE exactly). Typically, for such
inverse elliptic coefficient problems with finitely many unknowns, one can prove unique solvability and Lipschitz stability when sufficiently many measurements are used, cf.\ \cite{alberti2019calderon,harrach2019uniqueness,alberti2022infinite}.

\subsection{Monotonicity and convexity}

For ease of reference, we summarize some known properties of the forward operator $F:\ \R_{++}^m\to \sS^l$ \cite{harrach2021introduction,brojatsch2025required}. 

\begin{lemma}\label{lemma:Schur_argument}
For $y\in \R^m_{++}$, we have $F(y)\in \sS^l_{++}$, and, for all $Y\in \sS^l_{++}$, $Y-F(y)$ is the Schur complement of the upper left block in the block matrix
\[
M_y:=\begin{pmatrix} A_y & B^T\\ B & Y\end{pmatrix}\in \sS^{l+n}.
\]
Then, with $C:=B^T Y^{-1} B$,
\begin{align}\label{eq:det_Schur}
\det M_y= \det A_y \det (Y-F(y))=\det Y \det (A_y-C),
\end{align}
and the following equivalences hold
\begin{align*}
F(y)\preceq Y & \quad \text{ iff } \quad M_y\succeq 0
\quad \text{ iff } \quad A_y\succeq C, \text{ and }\\
F(y)\prec  Y & \quad \text{ iff } \quad M_y\succ 0
\quad \text{ iff } \quad A_y\succ C.
\end{align*}
Furthermore, $\dim\kernel(A_y-C)=\dim\kernel (Y-F(y))\leq l$, and
\[
\dim\kernel(A_y-C)=l\quad \text{ if and only if } \quad F(y)=Y.
\]
\end{lemma}
\begin{proof}
Since $A_y\succ 0$, and $B$ has full row rank, $F(y)\succ 0$.
The Schur complement of the upper left block in $M_y$ is $Y- BA_y^{-1} B^T=Y-F(y)$, and the Schur complement of the lower right block in $M_y$ is $A_y-B^T Y^{-1} B=A_y-C$. The Schur complement formula for the determinant gives \eqref{eq:det_Schur}. Also,
using $A_y\succ 0$ and $Y\succ 0$, Haynsworth inertia additivity yields that
\begin{alignat*}{3}
M_y \succeq 0 &\quad \text{if and only if}\quad & F(y) &\preceq Y \quad \text{if and only if}\quad A_y &\succeq C,\\
M_y \succ 0 &\quad \text{if and only if}\quad & F(y) &\prec Y \quad \text{if and only if}\quad A_y &\succ C,
\end{alignat*}
and that 
\begin{align*}
\dim \kernel (M_y)&= \dim \kernel (A_y)+\dim \kernel (Y-F(y))= \dim \kernel (Y) + \dim \kernel (A_y-C).
\end{align*}
Hence, the assertion follows.
\end{proof}

\begin{lemma}\label{lemma:F_properties}
$F:\ \R^m_{++}\to \sS^l$ has the following properties:
\begin{enumerate}[(a)]
\item $F$ is infinitely differentiable. Its first and second derivatives are given by
\begin{align}
F'(y)e_i&=-BA_y^{-1} A_i A_y^{-1} B^T\in \sS^l, \label{eq:F_deriv}\\
F''(y)(e_i,e_j)&= B A_y^{-1} A_i A_y^{-1} A_j A_y^{-1} B^T + B A_y^{-1} A_j A_y^{-1} A_i A_y^{-1} B^T\in \sS^l \label{eq:F_2nd_deriv}
\end{align}
for all $y\in \R^m_{++}$, and $i,j=1,\ldots,m$.
\item $F$ is monotonically non-increasing, i.e.
\begin{alignat*}{2}
F(z)&\preceq F(y) && \quad \text{ for all } y,z\in \R^m_{++} \text{ with } y\leq z,\\
F'(y)d&\preceq 0 && \quad \text{ for all } y\in \R^m_{++},\ d\in \R^m_{+}.
\end{alignat*} 
\item $F$ is convex, 
\begin{alignat*}{2}
F(ty+(1-t)z)&\preceq tF(y)+(1-t)F(z) \quad \text{ for all } y,z\in \R^m_{++},\ t\in[0,1],\\
F(z)&\succeq F(y)+F'(y)(z-y) \quad \text{ for all } y,z\in \R^m_{++}.
\end{alignat*} 
\item 
The function $y\mapsto F(y^{-1})$, $\R^m_{++}\to \sS^l$ is concave,
\begin{align}\label{eq:Ikehata}
F(z)&\succeq F(y)+F'(z)\left(\frac{z}{y}(z-y)\right) \quad \text{ for all } y,z\in \R^m_{++},
\end{align} 
where inversion and multiplication of vectors are interpreted componentwise. 
\item $F$ is positively homogeneous of degree $-1$,
\begin{equation}\label{eq:F_hom}
F'(y)y=-F(y) \quad \text{ and } \quad F(ty)=t^{-1} F(y) \quad \text{ for all } y\in \R^m_{++}, t>0.
\end{equation}
\item \label{lemma:F_properties_A}
Define $\mathcal A\in \LL(\sS^n,\R^m)$ by 
\[
\mathcal A(X):=\left( \langle A_j,X\rangle \right)_{j=1,\ldots,m}\in \R^m
\quad \text{ for all } X\in \sS^n.
\]
Then its adjoint $\mathcal A^*\in \LL(\R^m,\sS^n)$ fulfills
\[
\mathcal A^* d=\sum_{j=1}^m d_j A_j=A_d \quad \text{ for all } d\in \R^m.
\]
\item For all $y\in\R^m_{++}$, the adjoint of $F'(y)\in \LL(\R^m,\sS^l)$ is $F'(y)^*\in  \LL(\sS^l,\R^m)$, where
\[
F'(y)^*Z=-\mathcal A(A_y^{-1} B^T Z BA_y^{-1})\in \R^m \quad \text{ for all } Z\in \sS^l,
\]
i.e., for all $Z\in \sS^l$, $F'(y)^* Z=v\in \R^m$ fulfills
\[
v_i=\langle F'(y)e_i, Z \rangle=-\langle  A_i, A_y^{-1} B^T Z BA_y^{-1}\rangle, \quad i=1,\ldots,m.
\]
In particular
\[
F'(y)^*I=(\trace F'(y)e_i)_{i=1}^m=\nabla \trace F(y)\in \R^m.
\]
\end{enumerate}
\end{lemma} 
\begin{proof}
$F(y)=B A_y^{-1} B^T$ is a composition of linear mappings and inversions of a positive definite matrix. 
The assertions (a)--(c) thus follow from the fact that the matrix inversion map is infinitely differentiable, monotonically non-increasing, and convex on $\sS^n_{++}$. (d) can be shown via a completion-of-squares argument of Ikehata \cite{ikehata1998size}, cf.\ \cite[Lemma~2.1]{harrach2010exact} for a short proof for Neumann-to-Dirichlet operators in EIT that immediately carries over to this finite-dimensional setting. (e)--(g) immediately follow from the respective definitions and the cyclicity of the trace in the Frobenius inner products. 
\end{proof}

\begin{remark}\label{rem:resolvent}
For ease of reference, let us also record the following useful resolvent identity. It is readily verified by direct multiplication that two invertible matrices $M_1$ and $M_2$ satisfy
\begin{equation}\label{eq:resolvent}
M_2^{-1}-M_1^{-1}=
-M_1^{-1} (M_2-M_1) M_1^{-1}
+ M_1^{-1} (M_2-M_1) M_2^{-1} (M_2-M_1) M_1^{-1}.
\end{equation}
\end{remark}

\subsection{Semidefinite programming}\label{subsect:prelims_sdp}

We include the box constraints into the linear matrix inequality, and reverse signs, to obtain
\[
\Adm=\{ y\in [y^-,y^+]:\ A_y\succeq C\}
= \{ y\in \R^m:\ \tilde A_y\preceq \tilde C\},
\]
where $\tilde A_y:=\sum_{i=1}^m y_i \tilde A_i\in \sS^{\tilde n}$, $\tilde n:=n+2m$,
\begin{equation}\label{eq:tilde_A_and_C}
\tilde A_i:=\begin{pmatrix} -A_i & 0 & 0\\ 0 & \operatorname{diag}(e_i) & 0\\
0 & 0 & -\operatorname{diag}(e_i)
\end{pmatrix},\ \tilde C:=\begin{pmatrix} -C & 0 & 0\\ 0 & \operatorname{diag}(y^+) & 0\\
0 & 0 & -\operatorname{diag}(y^-)
\end{pmatrix}.
\end{equation}

For all $d\in \R^m$, the middle block of $\tilde A_d$ is $\operatorname{diag}(d)$ so that
\[
\tilde A_d=\sum_{j=1}^m d_j \tilde A_j=0 \quad \text{ implies } \quad d=0.
\]
Hence, the matrices $\tilde A_1,\ldots,\tilde A_m$ are linearly independent. 

Given $b\in \R^m$, we consider the semidefinite program (SDP)
\[
b^T y\to \text{max!} \quad \text{ s.t. } \quad \tilde A_y\preceq \tilde C.
\]

In the literature, this is commonly called the \emph{dual SDP} in canonical form. The corresponding \emph{primal SDP} is to find $\tilde X\in \sS^{\tilde n}$ that minimizes
\[
\langle \tilde C,\tilde X\rangle \to \text{min!} \quad \text{s.t.} \quad
\tilde X\succeq 0,\ \langle \tilde A_i,\tilde X\rangle=b_i\ \text{ for all }
i=1,\ldots,m.
\]
Weak duality always holds, i.e., for all primal feasible $\tilde X$, and all dual feasible $y$,
\[
\langle \tilde C,\tilde X\rangle-b^Ty=\langle \tilde C-\sum_{i=1}^m y_i \tilde A_i, \tilde X\rangle \geq 0.
\]

We first show existence of feasible points and strong duality. For $y\in \R^m_{++}$ we define
\[
\tilde X_y:=\begin{pmatrix} X_y & 0 & 0\\ 0 & 0 & 0\\ 0 & 0 & 0\end{pmatrix}\in \sS^{\tilde n}_+
\quad \text{ with } \quad X_y:= A_{y}^{-1} B^T B A_{y}^{-1}\in \sS^n_+
\]

\begin{lemma}\label{lemma:sdp_feasible}
\begin{enumerate}[(a)]
\item The point $y^\circ:=y^+-\frac{\delta}{2}\1$ is strictly dual feasible.
\item For $X\in \sS^n_+$ with $\mathcal A(X)=-b$, the zero extension $\tilde X\in \sS^{\tilde n}_+$ is primal feasible. Moreover, for a dual feasible $y\in \Adm$, we define, for $1>\epsilon>0$, 
\begin{align*}
\tilde X^\epsilon:= \begin{pmatrix} X+\epsilon I & 0 & 0\\ 0 &   \epsilon \operatorname{diag}(\1+{\mathcal A}(I)) & 0\\
0 & 0 &  \epsilon I\end{pmatrix}
\quad \text{ and } \quad
y^\epsilon:=(1-\epsilon) y+\epsilon y^\circ.
\end{align*}
Then $\tilde X^\epsilon$ is strictly primal feasible, and $y^\epsilon$ is strictly dual feasible.
\item If $b=F'(y)^*I$ for some $y\in \R^m_{++}$, then $\tilde X_y$
is primal feasible, and $\tilde X_y^\epsilon$ is strictly feasible. Moreover, strong duality holds between the primal and dual SDPs, their optimal values are attained and equal. 
\end{enumerate}
\end{lemma}
\begin{proof}
\begin{enumerate}[(a)]
\item Our assumption $\hat y\leq y^+-\delta\1$ and $A_\1\succ 0$ give that $A_{y^\circ}\succ A_{\hat y}\succeq C$. This shows (a).
\item Since $y$ is dual feasible, and $y^\circ$ is strictly dual feasible, the convex combination $y^\epsilon$ is also strictly dual feasible. Moreover, $e_i^T\mathcal{A}(I)=\trace(A_i)\geq 0$ gives $\tilde X^\epsilon\succ 0$, and
\begin{align*}
\langle \tilde A_i, \tilde X^\epsilon\rangle 
&= - \langle A_i, X +\epsilon I\rangle +  \epsilon  \langle \operatorname{diag}(e_i),  \operatorname{diag}(\1+{\mathcal A}(I))\rangle
 - \langle \operatorname{diag}(e_i), \epsilon I
\rangle\\
&= -\langle A_i,X \rangle - \epsilon \trace A_i +\epsilon +  \epsilon\,  (e_i^T\mathcal{A}(I))-\epsilon =
\langle \tilde A_i,\tilde X \rangle=b_i,
\end{align*}
for all $i=1,\ldots,m$. Thus $\tilde X^\epsilon$ is strictly primal feasible.
\item For every $i=1,\ldots,m$, trace cyclicity and \eqref{eq:F_deriv} give
\begin{align*}
\langle X_y,A_i\rangle=
\langle   BA_{y}^{-1} A_i A_{y}^{-1} B^T, I \rangle =-\langle F'(y)e_i,I\rangle=-b_i,
\end{align*}
so that, by (b), $\tilde X_y$ is primal feasible, and $\tilde X_y^\epsilon$ is strictly primal feasible.
Part (a) gives a strictly dual feasible point. Hence, standard SDP duality yields strong duality with attainment and equality of the optimal values, cf., e.g., \cite[Thm.~3.1]{vandenberghe1996semidefinite}.
\end{enumerate}
\end{proof}

Now we establish a connection between the inverse problem \eqref{eq:IP} and the SDP. 
\begin{lemma}\label{lemma:IP_vs_sdp}
Let $y\in \Adm$, and $b:=\nabla \trace F(y)$. Then, 
\begin{align}\label{eq:by_is_trace}
b^T y&=\langle \tilde A_y,\tilde X_y\rangle=-\langle A_y, X_y\rangle=-\trace F(y),\\
\langle \tilde C,\tilde X_y\rangle&=-\langle C,X_y\rangle=-\trace(F(y)\hat Y^{-1} F(y)).
\end{align}
The duality gap of the primal-dual feasible pair $(y,\tilde X_y)$ fulfills
\begin{align}\label{eq:dual_gap_trace}
\langle \tilde C,\tilde X_y\rangle - b^T y=\trace (F(y)-F(y)\hat Y^{-1} F(y))\leq \trace (\hat Y-F(y)),
\end{align}
and it vanishes if and only if $F(y)=\hat Y$.

Moreover, for $b=\nabla \trace F(\hat y)$, $\tilde X_{\hat y}$ is a primal optimal solution, and the inverse problem solution $\hat y$ is the unique solution of the dual SDP.
\end{lemma}
\begin{proof}
Let $y\in \Adm$, and $b:=\nabla \trace F(y)$. Feasibility of $\tilde X_y$, and trace cyclicity yield
\[
b^T y=\sum_{i=1}^m \langle \tilde A_i,\tilde X_y\rangle y_i
= \langle \tilde A_y,\tilde X_y\rangle=-\langle A_y, X_y\rangle
=-\langle B A_y^{-1} B^T, I\rangle=
-\trace F(y).
\]

$\tilde X_y$ is primal feasible by Lemma~\ref{lemma:sdp_feasible}.
Using trace cyclicity we obtain
\begin{align*}
\langle \tilde C,\tilde X_y\rangle  & =-\langle C,X_y\rangle
=-\langle  B^T \hat Y^{-1} B, A_y^{-1} B^T B A_y^{-1} \rangle\\
&=-\langle  F(y) \hat Y^{-1} F(y), I \rangle
=-\trace (F(y)\hat Y^{-1} F(y)).
\end{align*}

Using that $0\prec F(y)\preceq \hat Y$ implies $F(y)^{-1} \succeq \hat Y^{-1}$,
and applying the resolvent identity from Remark~\ref{rem:resolvent} gives 
\[
0\preceq F(y)-F(y)\hat Y^{-1} F(y)=F(y)(F(y)^{-1}-\hat Y^{-1})F(y)
\preceq \hat Y-F(y).
\]
This shows the last inequality in \eqref{eq:dual_gap_trace}, and also yields that
\[
\trace (F(y)-F(y)\hat Y^{-1} F(y))=0\quad \text{ iff } 
\quad F(y)=F(y)\hat Y^{-1} F(y) \quad \text{ iff } 
\quad F(y)=\hat Y.
\] 
Hence, for $b=\nabla \trace F(\hat y)$, the duality gap vanishes for $\tilde X_{\hat y}$ and $\hat y$, so that they are a pair of primal and dual solutions.

To prove uniqueness of the dual SDP solution for $b=\nabla \trace F(\hat y)$, let $y$ be any dual solution.
Then, by strong duality,
\begin{align*}
0&=\langle \tilde C,\tilde X_{\hat y}\rangle-b^Ty
=\langle \tilde C-\sum_{i=1}^m y_i \tilde A_i, \tilde X_{\hat y}\rangle = \langle  A_y-C,  X_{\hat y}\rangle.
\end{align*}
Since $A_y-C\succeq 0$, and $X_{\hat y}\succeq 0$, this gives $(A_y-C)X_{\hat y}=0$. By the full row rank of $B$, we have $\operatorname{rank} X_{\hat y}=l$, and therefore $\dim \kernel (A_y-C)\geq l$. With Lemma~\ref{lemma:Schur_argument}, it follows that $F(y)=\hat Y$, and unique solvability of the inverse problem yields that $y=\hat y$.
\end{proof}

\begin{remark}
Lemma \ref{lemma:IP_vs_sdp} motivates the following SDP-based iterative inverse problem solver for \eqref{eq:IP}. Given $y_k\in \Adm$, set $b:=\nabla \trace F(y_k)$, and obtain the next iterate $y_{k+1}$ from solving the dual SDP
\[
b^T y\to \text{max!} \quad \text{ s.t. } \quad \tilde A_y\preceq \tilde C.
\]
Lemma \ref{lemma:IP_vs_sdp} shows that $\hat y$ is a stationary point for this iteration.
Moreover, it also shows that $\tilde X_{y_k}$ is primal feasible, and it is 
primal optimal at the stationary point $\hat y$. 

This agrees with the Frank-Wolfe algorithm for
the constrained trace minimization problem \eqref{eq:trace_maximization} that we consider in this work. 
Also note that, for solving the SDP in the $(k+1)$-st Frank--Wolfe step, we will use $y_k^{\epsilon_k}$ and $\tilde X^{\epsilon_k}_{y_k}$ from Lemma~\ref{lemma:sdp_feasible} as starting points, with shifts $\epsilon_k\to 0$ that are adapted to the progress of iteration.
%
%
\end{remark}

\section{Globalized Frank-Wolfe trace minimization}

Let us summarize the main steps of our approach to recover $\hat y$ from $F(\hat y)=\hat Y$. Clearly, the solution $\hat y$ minimizes
\begin{equation}\label{eq:trace_maximization}
r(y):=\trace(\hat Y-F(y))\to \text{min!} \quad \text{subject to} \quad y\in [y^-,y^+],\ F(y)\preceq \hat Y.
\end{equation}
The mapping $F$ is convex with respect to the Loewner order, cf.\ Lemma~\ref{lemma:F_properties}, so that this is a concave minimization problem on a convex set. Moreover, a Schur complement argument shows, cf.\ Lemma \ref{lemma:Schur_argument}, that the non-linear semidefinite constraint can be rewritten as a linear matrix inequality (LMI)
\[
\Adm:=\{ y\in [y^-,y^+]:\ F(y)\preceq \hat Y\}=\{ y\in [y^-,y^+]:\ A_y\succeq C\},
\]
with $C:=B^T\hat Y^{-1}B$.

To minimize the residual $r$ on the admissible set $\mathcal Y_\text{adm}$
we apply an iterative Frank-Wolfe approach. Starting with a feasible initial point, we update the current iterate $y_k$ by setting $b:=\nabla \trace F(y_k)$ and then solve the dual SDP
\[
b^Ty\to \text{max!} \quad \text{subject to} \quad y\in [y^-,y^+],\ A_y\succeq C
\]
until the dual objective error is sufficiently decreased. For the latter, we monitor the primal-dual gap of the SDP solver and terminate once a prescribed fraction $0<\theta<1$ of the maximal linear improvement has been attained.

To ensure convergence to the true solution $\hat y$, the main challenge is to guarantee the avoidance of Frank-Wolfe stationary points $y_*\in \Adm$ where $y_*\neq \hat y$ but 
\[
(\nabla \trace F(y_*))^T (y-y_*)\leq 0 \quad \text{ for all $y\in \Adm$}.
\]
To avoid convergence to stationary points other than $\hat y$ we consider the SDP solution as failed, if the duality gap of an iterate $y_k$ falls below a certain fraction $0<\gamma<1$ of
the corresponding residual $r(y_k)$.

\subsection{Avoiding Frank-Wolfe stationary points}

We will first develop a criterion that locally excludes the existence of Frank-Wolfe stationary points other than $\hat y$.
We start by proving the following Loewner-order estimate for the first-order Taylor remainder of our forward function. By convexity, cf.\ Lemma \ref{lemma:F_properties},
\[
F(z)\succeq F(y)+F'(y)(z-y) \quad \text{ for all } y,z\in \R^m_{++},
\]
so that the first-order Taylor remainder is positive semidefinite. The following lemma shows that it can be bounded from above by a multiple of $F(y)$.

\begin{lemma}\label{lemma:Taylor_remainder}
For all $y,z\in \R^m_{++}$
\[
F(z)\preceq F(y)+F'(y)(z-y) + \omega(y,z) F(y),
\]
where $\omega(y,z):=\max_{i=1,\ldots,m} \frac{|z_i-y_i|^2}{y_i z_i}$. In particular, for all $y,z\in [y^-,y^+]$
\[
F(z)\preceq F(y)+F'(y)(z-y) + \frac{\norm{z-y}^2_\infty}{y_\sub{\min}^2} F(y),
\]
with $y_\sub{\min}:=\min_{i=1,\ldots,m} y^-_i$.
\end{lemma}
\begin{proof}
With $d:=z-y$, and $A_d=A_z-A_y$, we apply the resolvent identity from Remark~\ref{rem:resolvent} to obtain
\begin{align*}
F(z)-F(y)&=B ( A_z^{-1}-A_y^{-1}) B^T
= -B A_y^{-1} A_d A_y^{-1} B^T
+ B A_y^{-1} A_d A_z^{-1} A_d A_y^{-1} B^T\\
&=F'(y)d + B A_y^{-1} A_d A_z^{-1} A_d A_y^{-1} B^T.
\end{align*}
Thus, it remains to prove that $A_d A_z^{-1} A_d\preceq \omega A_y$.

By the definition of $\omega$, we have $|d_i|^2\leq \omega y_i z_i$ for all $i=1,\ldots,m$.
Therefore, for all $u,v\in \R^n$, the Cauchy-Schwarz inequality for the positive semidefinite bilinear form induced by $A_i\succeq 0$, followed by Young's inequality, yields
\begin{align*}
2 d_i u^T A_i v&\leq 2 |d_i| \sqrt{u^T A_i u} \sqrt{v^TA_i v}\leq 2\sqrt{\omega y_i u^T A_i u} \sqrt{z_i v^TA_i v}
\leq \omega y_i u^T A_i u + z_i v^T A_i v.
\end{align*}
Summing over $i=1,\ldots,m$, we obtain 
\[
2 u^T A_d v\leq \omega u^T A_y u + v^T A_z v \quad \text{ for all } u,v\in \R^n.
\]
Setting $v:=A_z^{-1} A_du$, it follows that
\[
2 u^T A_d A_z^{-1} A_d u\leq \omega u^T A_y u + u^T A_d A_z^{-1} A_d u \text{ for all } u\in \R^n,
\]
which gives $A_d A_z^{-1} A_d\preceq \omega A_y$ and proves the assertion.
\end{proof}

This allows us to characterize a neighborhood of the inverse problem solution $\hat y$ in which $\hat y$ is the only Frank-Wolfe stationary point.
\begin{lemma}\label{lemma:no_stat_points}
For all $y\in [y^-,y^+]$ with $F(y)\preceq \hat Y=F(\hat y)$
\begin{align}\label{eq:grad_estimate_y}
(\nabla \trace F(y))^T (\hat y-y) 
&\geq \left( 1- \frac{L}{y_\sub{\min}^2} \norm{\hat y-y}_\infty \trace \hat Y \right) \trace (F(\hat y)- F(y)).
\end{align}
Hence, if 
\[
\norm{\hat y-y}_\infty < \frac{y_\sub{\min}^2}{L\trace\hat Y}, \quad \text{ and } \quad y\neq \hat y,
\]
then $y$ is not a Frank-Wolfe stationary point,
\[
(\nabla \trace F(y))^T (\hat y -y)>0.
\]
\end{lemma}
\begin{proof}
We use the Lipschitz stability assumption \eqref{eq:Lipschitz_stability}, and $F(y)\preceq F(\hat y)$, 
to obtain  
\begin{equation}\label{eq:Use_Lipschitz}
\norm{\hat y-y}_\infty \leq L \norm{F(\hat y)-F(y)}_2\leq L (\trace(F(\hat y)-F(y))).
\end{equation}
Applying Lemma~\ref{lemma:Taylor_remainder} yields
\begin{align*}
\trace (F(\hat y)- F(y)) &\leq (\nabla \trace F(y))^T (\hat y-y) +\frac{\norm{\hat y-y}_\infty^2}{y_\sub{\min}^2} \trace F(y)\\
&\leq (\nabla \trace F(y))^T (\hat y-y) +\frac{L}{y_\sub{\min}^2} \norm{\hat y-y}_\infty (\trace (F(\hat y)-F(y))) \trace \hat Y,
\end{align*}
which gives \eqref{eq:grad_estimate_y}. Hence, for $y\neq \hat y$,
\[
\norm{\hat y-y}_\infty < \frac{y_\sub{\min}^2}{L\trace\hat Y} \quad \text{ implies } 
\quad
 \frac{L}{y_\sub{\min}^2} \norm{\hat y-y}_\infty \trace \hat Y<1,
\]
and thus $(\nabla \trace F(y))^T (\hat y-y)>0$.
\end{proof}

The neighborhood without spurious stationary points can also be characterized by the trace values of the iterates. The following corollary is the basis of our convergence results. 
\begin{corollary}\label{cor:no_stat_points}
For all $y\in [y^-,y^+]$ with $F(y)\preceq \hat Y=F(\hat y)$
\begin{align}\label{eq:grad_estimate_trace}
(\nabla \trace F(y))^T (\hat y-y) 
&\geq \left( 1- \frac{L^2}{y_\sub{\min}^2} \trace \left( F(\hat y)-F(y) \right) \trace \hat Y \right) \trace (F(\hat y)- F(y)).
\end{align}
Hence, if 
\[
\trace F(y)> \trace (\hat Y)-\frac{y_\sub{\min}^2}{L^2\trace{\hat Y}}, \quad \text{ and } \quad y\neq \hat y,
\]
then $y$ is not a Frank-Wolfe stationary point,
\[
(\nabla \trace F(y))^T (\hat y -y)>0.
\]
\end{corollary}
\begin{proof}
The assertion \eqref{eq:grad_estimate_trace} immediately follows from \eqref{eq:grad_estimate_y}
by using \eqref{eq:Use_Lipschitz}.
\end{proof}

\subsection{Local and global convergence against the true solution}

We are now ready to quantify how much the Frank-Wolfe iterates decrease the trace residual $r(y_k)=\trace (\hat Y- F(y_k))$, and to prove convergence to the true inverse problem solution $\hat y$.

The following theorem shows that the SDP solver can be stopped as soon as it has attained a prescribed fraction
$\theta$ of the maximal possible increase in the dual objective value. 
It also provides a duality-gap criterion that will prevent the iterates from approaching Frank-Wolfe stationary points other than $\hat y$.

\begin{theorem}\label{thm:Frank_Wolfe_Convergence}
Let $\gamma,\theta\in (0,1)$, $y_{k}\in \Adm$, $b:=\nabla \trace F(y_k)$, and let $\overline z\in\R$ be an upper bound of the optimal dual objective value, i.e.
\[
\overline z\geq \max_{w\in \Adm} b^T w.
\]

For all $y_{k+1}\in \Adm$, it holds that
\begin{equation}\label{eq:FW_monotonicity}
r(y_{k+1})\leq r(y_k)-b^T (y_{k+1}-y_k).
\end{equation}
Hence, increasing the dual objective reduces the trace functional.

If $y_k\in \Adm$ satisfies
\begin{align}\label{eq:trace_condition_convergence}
r(y_k)<(1-\gamma) \frac{y_\sub{\min}^2}{L^2 \trace\hat Y},
\end{align}
then 
\begin{equation}\label{eq:gamma_criterion}
\overline z - b^T y_k\geq  \gamma r(y_k).
\end{equation}
If, additionally, $y_{k+1}\in \Adm$ satisfies
\begin{equation}\label{eq:stop_crit_sdp}
b^T (y_{k+1}-y_k) \geq \theta (\overline z - b^T y_k), 
\end{equation}
then, with $\mu_k:=1-\frac{L^2}{y_\sub{\min}^2}  r(y_k) \trace \hat Y$,
\begin{equation}\label{eq:traceres_decrease}
r(y_{k+1})\leq (1-\theta \mu_k) r(y_k) \leq (1-\theta \gamma) r(y_k).
\end{equation}
\end{theorem}
\begin{proof}
By convexity,
\begin{align*}
\trace F(y_{k+1})&\geq \trace F(y_k)+(\nabla \trace F(y_k))^T(y_{k+1}-y_k)=\trace F(y_k)+b^T (y_{k+1}-y_k),
\end{align*}
which proves \eqref{eq:FW_monotonicity}.

Using $\overline z\geq b^T\hat y$, $b=\nabla \trace F(y_k)$, and Corollary \ref{cor:no_stat_points} gives, 
\begin{equation}\label{eq:dual_gap_vs_mu_r}
\overline z-b^Ty_k \geq b^T(\hat y-y_k)\geq \left( 1- \frac{L^2}{y_\sub{\min}^2}  r(y_k) \trace \hat Y \right) r(y_k)=\mu_k r(y_k)\geq \gamma r(y_k).
\end{equation}
This shows \eqref{eq:gamma_criterion}. 

If, additionally, \eqref{eq:stop_crit_sdp} holds, then we obtain from \eqref{eq:FW_monotonicity} and \eqref{eq:dual_gap_vs_mu_r} that
\[
r(y_{k+1})\leq r(y_k)-b^T (y_{k+1}-y_k)\leq r(y_k)-\theta (\overline z - b^T y_k)
\leq (1-\theta\mu_k) r(y_k) \leq (1-\theta\gamma) r(y_k).
\]
This finishes the proof.
\end{proof}

\begin{corollary}[Local and global convergence results]\label{cor:local_global_convergence}
Under the assumption that the algorithms terminate in finitely many steps, the following local and global convergence results hold.
\begin{enumerate}[(a)]
\item Theorem \ref{thm:Frank_Wolfe_Convergence} shows local linear convergence of the Frank-Wolfe method. 
If the initial value $y_0\in \Adm$ fulfills \eqref{eq:trace_condition_convergence},
and, in each Frank-Wolfe iteration, the SDP is solved until \eqref{eq:stop_crit_sdp} is fulfilled, then this either yields $y_k=\hat y$ after finitely many steps,
or a sequence $(y_k)_{k\in \N_0}$ with $r(y_k)\to 0$, and
\[
\sup_{k\in \N}\frac{r(y_{k+1})}{r(y_{k})}\leq 1-\theta\gamma<1, \quad \text{ and } \quad
\quad \limsup_{k\to \infty} \frac{r(y_{k+1})}{r(y_{k})}\leq 1-\theta<1.
\]
By Lipschitz stability this also yields linear convergence of $(y_k)_{k\in \N_0}\to \hat y$.
\item Theorem \ref{thm:Frank_Wolfe_Convergence} also permits rigorous globalization. Fix $0<\gamma<1$ and choose
$1>\kappa>1-\theta\gamma$. In the $k$-th step of the globalized algorithm, we run the SDP solver until it either returns a candidate $y^{\operatorname{cand}}_{k+1}\in\Adm$ 
that satisfies
\begin{equation}\label{eq:stop_crit_sdp_cand}
b^T (y^{\operatorname{cand}}_{k+1}-y_k) \geq \theta (\overline z - b^T y_k), 
\end{equation}
or abort and report failure if the $\gamma$-criterion \eqref{eq:gamma_criterion} is violated.

If no failure is reported, then we check the candidate for global convergence condition
\begin{equation}\label{eq:check_candidate}
r(y^{\operatorname{cand}}_{k+1})\leq \kappa r(y_k).
\end{equation}
If this is violated, or failure is reported, then we use a global search strategy to find 
$y_{k+1}\in \Adm$ that fulfills \eqref{eq:check_candidate}.

Then, for every initial value $y_0\in \Adm$, this method yields linear convergence 
\[
r(y_k)\to 0, \quad \text{ and } \quad y_k\to \hat y,
\]
and only finitely many iterations will require global search steps. 

\end{enumerate}
\end{corollary}
\begin{proof}
We only need to consider the case where $y_k\neq \hat y$ for all $k\in\N_0$.

Equations \eqref{eq:FW_monotonicity} and \eqref{eq:traceres_decrease} show, by induction, 
that $r(y_k)$ is monotonically decreasing, and $\mu_k$ is monotonically increasing with $\mu_k>\mu_0>\gamma$. Hence, $r(y_k)\to 0$ and $\mu_k\to 1$, which also proves the asserted asymptotic rate. 

For (b), the acceptance test \eqref{eq:check_candidate} or the global search gives $r(y_{k+1})\leq\kappa r(y_k)$ in every iteration. Thus $r(y_k)\to 0$ and, by Lipschitz stability, $y_k \to \hat y$. Consequently, \eqref{eq:trace_condition_convergence} holds for all sufficiently large $k\in\N$.
Theorem~\ref{thm:Frank_Wolfe_Convergence} then shows that the $\gamma$-criterion will stay fulfilled, and that every candidate satisfying
\eqref{eq:stop_crit_sdp_cand} fulfills
\[
r(y^{\operatorname{cand}}_{k+1}) \leq(1-\theta \gamma)r(y_k) <\kappa r(y_k).
\]
Hence, every sufficiently late SDP call returns a candidate that passes the acceptance test \eqref{eq:check_candidate}, so that only finitely many global searches are required.
\end{proof}

We summarize the globalized algorithm in Algorithm \ref{algo:full_algo_outline}. 
The details of the SDP solver and the global search are given in the next two subsections, where we also prove that they terminate after finitely many steps, and that the inner iteration number of the SDP solver will stay uniformly bounded over the outer Frank-Wolfe loop. This gives asymptotic linear convergence of the algorithm with respect to its runtime.
For Algorithm \ref{algo:full_algo_outline}, note that the SDP solver will require a strictly dual feasible point $y^\circ$, which we obtain from Lemma \ref{lemma:sdp_feasible}.
This is also a natural starting point for the method, cf.\  Theorem~\ref{thm:FW_with_monotonicity}.

\begin{algorithm}
\caption{Globalized solver for inverse elliptic coefficient problems}
\label{algo:full_algo_outline}
\begin{algorithmic}
\State \textbf{Input:} $A_1,\ldots,A_m\in \sS^n_+$, $B\in \R^{l\times n}$, $y^-,y^+\in \R^m_{++}$, $\delta>0$, $\hat Y=F(\hat y)\in \sS^l_{++}$.
\Statex \vspace{-2ex}
\Statex \textit{\% Initialize and choose algorithm design parameters:}
\State \textbf{set} $C:=B^T \hat Y^{-1} B$, $\Adm:=\{y\in [y^-,y^+],\ A_y\succeq C\}$, and $y^\circ:=y^+-\frac{\delta}{2}\1$
\State \textbf{choose} initial value $y_0\in \Adm$ \hfill \textit{\% e.g., $y_0:=y^\circ$, cf.\ Thm.~\ref{thm:FW_with_monotonicity}}
\State \textbf{choose} parameters $0<\theta<1$, $0<\gamma<1$, and $(1-\theta\gamma)<\kappa<1$
\Statex \vspace{-2ex}
\Statex \textit{\% Outer loop: Frank-Wolfe iteration for $r(y):=\trace (\hat Y- F(y))\to \text{min!}$ s.t. $y\in \Adm$:}

\For{$k=0,1,2,\ldots$}
\Statex \vspace{-2ex}
\State \textbf{if} $r(y_k)=0$ \textbf{terminate} with $y_k$; \textbf{ end if}
\Statex \vspace{-2ex}
\State \textbf{set} $b:=-( \trace (B A_{y_k}^{-1} A_i A_{y_k}^{-1} B^T) )_{i=1}^m\in \R^m$
\State \textbf{solve sdp} $b^T y\to \text{max!}$ \ s.t.\ $y\in \Adm$
\State \textbf{stop sdp solver} with \textbf{failure} or with $y^\text{cand}_{k+1}\in \Adm$, that fulfills\\ 
\centerline{$
\displaystyle b^T (y^\text{cand}_{k+1}-y_k) \geq \theta\, \max_{w\in \Adm} b^T (w-y_k)
$}
\Statex \vspace{-2ex}
\State \textit{\% Possibly replace Frank-Wolfe iterate to ensure global convergence:}

\If{sdp solver returns \textbf{failure} or $r(y^\text{cand}_{k+1})> \kappa r(y_k)$} 
 \State \textbf{global search} for $r(y)\to \text{min!}$ s.t.\ $y\in \Adm$
 \State \textbf{stop global search} with $y_{k+1}$ that fulfills $r(y_{k+1})\leq \kappa r(y_k)$
 \Else 
 \State \textbf{set} $y_{k+1}:=y^\text{cand}_{k+1}$
 \EndIf
 \EndFor
\Statex \vspace{-2ex}
\State \textbf{Output:} Algorithm terminates with $\hat y$, or constructs sequence $(y_k)_{k\in \N_0}\to \hat y$
\end{algorithmic}
\end{algorithm}

\begin{remark}\label{rem:fast_combi}
Our algorithm can easily be combined with faster local solvers. In each outer iteration, we may first compute a feasible candidate with such a solver and accept it if it satisfies \eqref{eq:check_candidate}. If it fails, we try the Frank-Wolfe step, and if this also fails, then we resort to the global search. Thus, the algorithm keeps its global convergence properties, while inheriting the asymptotic convergence speed of the faster solver.

Likewise, the global search can utilize a share of its runtime budget to run a faster local solver from random starting points. If this yields a feasible iterate that fulfills \eqref{eq:check_candidate}, then the computationally expensive global search can be terminated early. 
\end{remark}

We finish this section with a stronger version of Theorem~\ref{thm:Frank_Wolfe_Convergence} under an additional monotonicity condition. It motivates $y_0:=y^\circ$ as a suggested choice for Algorithm \ref{algo:full_algo_outline}. We will comment on the suggested choice for the other parameters, $\theta$, $\kappa$, and $\gamma$ at a later point.

\begin{theorem}\label{thm:FW_with_monotonicity}
Let $\theta\in (0,1)$, $y_{k}\in \Adm$, $b:=\nabla \trace F(y_k)$, $\overline z\geq \max_{w\in \Adm} b^T w$, and
\[
q:=\max_{i=1,\ldots,m} \frac{y_{k,i}}{y^-_i},
\]
where $y_{k,i}=e_i^T y_k$ denotes the $i$-th entry of the vector $y_k\in \R^m$.
\begin{enumerate}[(a)]
\item If $y_{k}\geq \hat y$, then 
\begin{equation}\label{eq:gamma_criterion_mon}
\overline z - b^T y_k\geq \frac{1}{q} r(y_k).
\end{equation}
\item If $y_k\in \Adm$ satisfies \eqref{eq:gamma_criterion_mon}, and $y_{k+1}\in \Adm$ satisfies \eqref{eq:stop_crit_sdp}, then
\begin{equation}\label{eq:traceres_decrease_mon}
r(y_{k+1})\leq \left( 1-\frac{\theta}{q} \right) r(y_k).
\end{equation}
\end{enumerate}
\end{theorem}
\begin{proof}
The monotonicity condition $y_k\geq \hat y$ gives that
\[
\frac{y_k}{\hat y}(y_k-\hat y)\leq q (y_k-\hat y).
\]
By Lemma~\ref{lemma:F_properties}, it follows that
\begin{align*}
 F(\hat y)-F(y_k) &\preceq -F'(y_k)\left(\frac{y_k}{\hat y}(y_k-\hat y)\right)
 \preceq - q F'(y_k) (y_k-\hat y),
\end{align*}
and thus
\begin{align*}
(\nabla \trace F(y_k))^T (\hat y -y_k)\geq \frac{1}{q}r(y_k).
\end{align*}
We now use the same argument as in Theorem~\ref{thm:Frank_Wolfe_Convergence}. $b^T\hat y\leq \overline z$ gives
\begin{equation*}
\overline z-b^Ty_k \geq b^T(\hat y-y_k)\geq  \frac{1}{q} r(y_k),
\end{equation*}
which shows (a). Convexity gives
\[
\trace F(y_{k+1})\geq \trace F(y_k)+(\nabla \trace F(y_k))^T(y_{k+1}-y_k)=\trace F(y_k)+b^T (y_{k+1}-y_k),
\]
so that (b) follows from \eqref{eq:gamma_criterion_mon} and \eqref{eq:stop_crit_sdp}.
\end{proof}

\subsection{The SDP solver}\label{subsect:sdp_solver}
We use the dual-scaling method of Benson, Ye, and Zhang \cite{benson2000solving}.
We summarize the details needed for the Frank-Wolfe method and prove that its computational cost remains uniformly bounded if each iteration is initialized with a centralization step. 

We include the box constraints into the LMI, and reverse signs as in subsection \ref{subsect:prelims_sdp} to obtain 
\[
\Adm= \{ y\in \R^m:\ \tilde A_y\preceq \tilde C\},
\]
with matrices $\tilde A_1,\ldots, \tilde A_m,\tilde C\in \sS^{\tilde n}$ as given in \eqref{eq:tilde_A_and_C}.
Throughout this subsection, we also write
\[
\tilde{\mathcal A}(\tilde X):=\left( \langle \tilde A_i,\tilde X\rangle \right)_{i=1}^m \in \R^m
\quad \text{ for } \tilde X\in \sS^{\tilde n},
\]
For $y\in\Adm$, we also write $S_y:=A_y-C\in \sS^n_+$, and $\tilde S_y:=\tilde C-\tilde A_y\in \sS^{\tilde n}_+$.

Let $y_k\in \Adm$ denote the current iterate of the Frank-Wolfe method, and $r(y_k)\neq 0$.
As shown in Subsection \ref{subsect:prelims_sdp}, the zero extension $\tilde X_{y_k}\in \sS^{\tilde n}_+$ of $X_{y_k}:=A_{y_k}^{-1} B^T B A_{y_k}^{-1}$ is primal feasible. Using the strictly feasible point $y^\circ$, the primal-dual feasible pair $(y_k,\tilde X_{y_k})$ can be shifted to a strictly feasible pair $\tilde X_{y_k}^{\epsilon}$ and $y_k^{\epsilon}$ by Lemma~\ref{lemma:sdp_feasible}.

We fix parameters
\begin{equation}\label{eq:dual_scaling_parameters}
\rho>\tilde n+\sqrt{\tilde n}=n+2m+\sqrt{n+2m},\quad \text{ and } \quad \alpha:=0.4,
\end{equation}
and set the centralization shift to
\begin{align}\label{eq:dual_scaling_shift}
\epsilon_k&:=\min\{r(y_k), \textstyle \frac{1}{2}\}
\end{align}
and start the dual scaling iteration with $\tilde X_0:=\tilde X_{y_k}^{\epsilon_k}$,
\begin{align}\label{eq:dual_scaling_start1}
\eta_0&:=y_k^{\epsilon_k}=(1-\epsilon_k)y_k+\epsilon_k y^\circ,\\
\label{eq:dual_scaling_start2}
\overline z_0&:=\langle \tilde C,\tilde X_0\rangle
=-\trace (F(y_k)\hat Y^{-1} F(y_k))+\epsilon_k\left(\norm{y^+-y^-}_1+\trace S_{y^+}\right).
\end{align}
Thus $\eta_0$ is strictly feasible and $\overline z_0\in \R$ is an upper bound for the dual objective. 

Starting from these values, we iterate over $j\in \N_0$ as in Benson, Ye, and Zhang \cite{benson2000solving}. We calculate
\begin{align}
\nonumber M_{\eta_j}&:=\left( \langle \tilde A_r \tilde S_{\eta_j}^{-1}, \tilde S_{\eta_j}^{-1} \tilde A_s \rangle \right)_{r,s=1}^m\\
\label{eq:SDP_dualscaling1}
&=\left( \langle  A_r  S_{\eta_j}^{-1},  S_{\eta_j}^{-1}  A_s \rangle \right)_{r,s=1}^m
+\operatorname{diag}((y^+-\eta_j)^{-2}+(\eta_j-y^-)^{-2})
\in \sS^{m}_{++},\\
\label{eq:SDP_dualscaling2}
\tilde{\mathcal A}(\tilde S_{\eta_j}^{-1})&=-\mathcal A(S_{\eta_j}^{-1})+(y^+-\eta_j)^{-1}-(\eta_j-y^-)^{-1}\in \R^m,\\
\label{eq:SDP_dualscaling3}
p_{\eta_j,\overline z_j}&:=-\frac{\rho}{\overline z_j-b^T \eta_j}b+\tilde{\mathcal A}(\tilde S_{\eta_j}^{-1})\in \R^m,\\
\label{eq:SDP_dualscaling4} d_{\eta_j,\overline z_j}&:=-M_{\eta_j}^{-1} p_{\eta_j,\overline z_j}\in \R^m,
\end{align}
where the negative powers of vectors in \eqref{eq:SDP_dualscaling1} and \eqref{eq:SDP_dualscaling2} are interpreted componentwise.

Then we check, with $\tilde n=n+2m$, whether 
\begin{align}\label{eq:SDP_eta_or_z}
\sqrt{  p^T_{\eta_j,\overline z_j} M_{\eta_j}^{-1} p_{\eta_j,\overline z_j} }<\min \left\{ \alpha\sqrt{\frac{\tilde n}{\tilde n+\alpha^2}},1-\alpha\right\}.
\end{align}
If \eqref{eq:SDP_eta_or_z} holds, we keep the dual variable and update the bound by setting
\begin{align}\label{eq:SDP_update_z}
\eta_{j+1}:=\eta_j, \quad \text{ and } \quad
\overline z_{j+1}:=b^T \eta_j+\frac{\overline z_j-b^T \eta_j}{\rho} \left( d^T_{\eta_j,\overline z_j}  \tilde{\mathcal A}(\tilde S_{\eta_j}^{-1}) +\tilde n \right).
\end{align}
Otherwise, we keep the bound and update the dual variable by setting
\begin{align}\label{eq:SDP_update_eta}
\overline z_{j+1}=\overline z_j, \quad \text{ and } \quad \eta_{j+1}:=\eta_j+\frac{\alpha}{\sqrt{ p^T_{\eta_j,\overline z_j} M_{\eta_j}^{-1} p_{\eta_j,\overline z_j} }}d_{\eta_j,\overline z_j}.
\end{align}

We stop the method when $\eta_{j+1}$ and $\overline z_{j+1}$ fulfill
\begin{equation}\label{eq:stp_success_test}
b^T (\eta_{j+1}-y_k)\geq \theta (\overline z_{j+1}-b^T y_k).
\end{equation}
In that case we return $y_{k+1}^\text{cand}:=\eta_{j+1}$ to the outer solver. If this has not occurred and
\begin{equation}\label{eq:sdp_failure_test}
\overline z_{j+1}-b^Ty_k<\gamma r(y_k),
\end{equation}
then we stop the solver and report failure to the outer solver.

\begin{algorithm}
\caption{SDP solver subroutine based on Benson-Ye-Zhang dual scaling method}
\label{algo:sdp_solver}
\begin{algorithmic}
\State \textbf{Input:} $A_1,\ldots,A_m\in \sS^n_+$, $B\in \R^{l\times n}$, $y^-,y^+\in \R^m_{++}$, $\hat Y\in \sS^l_{++}$, 
$\theta,\gamma\in (0,1)$,
\State \phantom{\textbf{Input:}} Outer iterate $y_k\in \Adm\setminus \{\hat y\}$, strictly feasible $y^\circ\in \Adm$, $b=\nabla \trace F(y_k)$
\Statex \vspace{-2ex}
\State \textbf{set} algorithm parameters as in \eqref{eq:dual_scaling_parameters}, and initialize $\overline z_0$ and $\eta_0$ by \eqref{eq:dual_scaling_shift}--\eqref{eq:dual_scaling_start2}
\For{$j=0,1,2,\ldots$}
%
\Statex \vspace{-2ex}
\State \textbf{calculate} $M_{\eta_j}$, $p_{\eta_j,\overline z_j}$, and $d_{\eta_j,\overline z_j}$ by 
\eqref{eq:SDP_dualscaling1}--\eqref{eq:SDP_dualscaling4}
\Statex \vspace{-2ex}
\If{\eqref{eq:SDP_eta_or_z} holds}
\State \textbf{update} $\overline z_{j+1}$ as in \eqref{eq:SDP_update_z}, and keep $\eta_{j+1}:=\eta_j$
\Else
\State \textbf{update} $\eta_{j+1}$ as in \eqref{eq:SDP_update_eta}, and keep $\overline z_{j+1}:=\overline z_{j}$
\EndIf
\Statex \vspace{-2ex}

\If{$b^T (\eta_{j+1}-y_k)\geq \theta (\overline z_{j+1}-b^T y_k)$}
    \State \Return $y_{k+1}^\text{cand}=\eta_{j+1}$
\ElsIf{$\overline z_{j+1}-b^T y_k<\gamma r(y_k)$}
    \State \Return \textbf{failure}
\EndIf
\EndFor
\Statex \vspace{-2ex}
\State \textbf{Output:} $y_{k+1}^\text{cand}$ fulfills \eqref{eq:stop_crit_sdp_cand}, or \textbf{failure} ($\max_{w\in\Adm} b^T(w-y_k)<\gamma r(y_k)$)
\end{algorithmic}
\end{algorithm}

We summarize the complete sdp solver subroutine without repeating formulas in Algorithm~\ref{algo:sdp_solver}.
Note that the update of the upper bound \eqref{eq:SDP_update_z} is equivalent to defining
the primal feasible
\begin{equation}\label{eq:dualscaling_update_X}
\tilde X_{j+1}:=\frac{\overline z_j-b^T \eta_j}{\rho} \tilde S_{\eta_j}^{-1} \left( \tilde{\mathcal A}^* d_{\eta_j,\overline z_j} +\tilde S_{\eta_j} \right) \tilde S_{\eta_j}^{-1}\in \sS^{\tilde n},
\end{equation}
and setting $\overline z_{j+1}=\langle \tilde C,\tilde X_{j+1}\rangle$, cf.\ \cite{benson2000solving}. This also yields, for all $j\in \N_0$,
\begin{equation}\label{eq:dualscaling_XS}
\langle \tilde X_j,\tilde S_{\eta_j}\rangle=\overline z_j - b^T\eta_j.
\end{equation}
However, the primal feasible sequence $(\tilde X_{j})_{j\in \N_0}$ is only required for the analysis of the method. The numerical algorithm can be implemented without explicitly computing this primal variable. This is particularly well suited to our setting, where only the dual solution is required and the primal variables are very high-dimensional. Also, updating $\overline z_j$ by \eqref{eq:SDP_update_z} is comparatively cheap as the matrix $M$ and its factorization do not have to be recomputed.

\begin{lemma}\label{lemma:sdp_terminates}
 For each $k\in\N_0$ with $y_k\neq\hat y$, there exists $N_k\in \N$, so that
Algorithm~\ref{algo:sdp_solver} terminates after $N_k$ steps. 
It either returns a dual feasible candidate
$y_{k+1}^\text{cand}=\eta_{N_k}\in\Adm$ that fulfills \eqref{eq:stop_crit_sdp_cand} with $\overline z=\overline z_{N_k}$, 
or it returns failure. If it returns failure then $y_k$ violates \eqref{eq:gamma_criterion}.
\end{lemma}
\begin{proof}
Benson, Ye and Zhang \cite{benson2000solving} prove that, in each step, $M\in \sS^{m}_{++}$, 
that the primal and dual iterates, $\tilde X_{j}$ and $\eta_j$, are strictly feasible,
$\overline z_j=\langle \tilde C,\tilde X_j\rangle$ is an upper bound for the dual objective, and 
\[
\overline z_j-b^T \eta_j\to 0 \quad \text{ for } j\to \infty.
\]

If the failure test \eqref{eq:sdp_failure_test} never holds, then
$\overline z_j-b^T y_k\geq \gamma r(y_k)>0$ for all $j\in \N$. Hence,
\[
\frac{\overline z_j-b^T \eta_j }{\overline z_j-b^T y_k}\to 0,
\]
so that after finitely many steps, 
\begin{equation*}
\overline z_j-b^T \eta_j\leq (1-\theta) (\overline z_j-b^T y_k).
\end{equation*}
If this is fulfilled after $N_k$ steps, then 
\[
b^T (\eta_{N_k}-y_k)=(\overline z_{N_k}-b^T y_k)-(\overline z_{N_k}-b^T \eta_{N_k})
\geq \theta (\overline z_{N_k}-b^T y_k)\geq \theta \max_{w\in \Adm} b^T (w-y_k),
\]
and the algorithm terminates with $y_{k+1}^\text{cand}$ satisfying \eqref{eq:stop_crit_sdp_cand}.

If the algorithm returns failure after $N_k$ steps, then
\[
\max_{w\in\Adm}b^Tw-b^Ty_k \leq \overline z_{N_k}-b^Ty_k <\gamma r(y_k).
\]
\end{proof}

The remainder of this subsection is devoted to proving that the number of iteration steps of the dual scaling iteration stays uniformly bounded with respect to $k\in \N$.
We extract the following explicit bound
for the duality gap after $j$ iterations from the convergence proof in \cite{benson2000solving}.

\begin{lemma}\label{lemma:dual_scaling_convergence_speed}
\begin{enumerate}[(a)]
\item Let $\Psi$ denote the Tanabe-Todd-Ye primal-dual potential function 
\begin{equation}\label{eq:Tanabe}
\Psi(\tilde X,\tilde S):=\rho \ln (\langle \tilde X,\tilde S \rangle ) - \ln \det \tilde X - \ln \det\tilde S\quad \text{ for all } \tilde X,\ \tilde S\in \sS^{\tilde n}_{++}.
\end{equation}
The iterates of the dual scaling method fulfill
\[
\Psi (\tilde X_{j+1}, \tilde S_{\eta_{j+1}})\leq \Psi(\tilde X_j,\tilde S_{\eta_j})-\frac{1}{50}.
\]
\item Let $\tilde X_0$ and $\eta_0$ be strictly feasible for the primal and dual problem, respectively. Then the iterates of the dual scaling method as defined above, with $\overline z_0:=\langle \tilde C, \tilde X_0\rangle$, fulfill
\begin{align}\label{eq:dual_scaling_conv_speed_bound}
 \overline z_j - b^T \eta_j 
&\leq  \frac{( \overline z_0 - b^T \eta_0 )^\frac{\rho}{\rho-\tilde n}}{   
(\tilde n^{\tilde n} \det \tilde X_0 \det \tilde S_{\eta_0})^\frac{1}{\rho-\tilde n}}  e^{-\frac{1}{50(\rho-\tilde n)}j}.
\end{align}
\end{enumerate}
\end{lemma}
\begin{proof}
(a) is proven in \cite[Thm.~2]{benson2000solving}. Using the arithmetic-geometric mean inequality as in the first lines in the proof of \cite[Cor.~1]{benson2000solving} gives
\begin{align*}
(\rho-\tilde n)\ln \left( \langle \tilde C, \tilde X_j \rangle - b^T \eta_j \right)
&=(\rho-\tilde n) \ln \langle \tilde X_j,\tilde S_{\eta_j} \rangle 
\leq  \Psi(\tilde X_j, \tilde S_{\eta_j})-\tilde n\ln \tilde n\\
&\leq \Psi(\tilde X_0, \tilde S_{\eta_0})-\tilde n\ln \tilde n -\frac{j}{50}.
\end{align*}
Hence, (b) follows by \eqref{eq:dualscaling_XS}.
\end{proof}

We now use Lemma~\ref{lemma:dual_scaling_convergence_speed} to asymptotically estimate the required number of iterations. Note that the constants in the following lemma only contain easily computable quantities, so that the asymptotic bound of the required number of dual scaling iterations can be explicitly calculated.

\begin{lemma}\label{lemma:dual_scaling_bound_iterates}
Assume that Algorithm~\ref{algo:sdp_solver} is started with $y_k$, $k\in \N_0$, where $y_k\to \hat y$. 
Then, for sufficiently large $k\in \N_0$, Algorithm~\ref{algo:sdp_solver} terminates after $N_k$ iteration steps 
with a dual feasible candidate
$\eta_{N_k}\in\Adm$ that fulfills \eqref{eq:stop_crit_sdp}, with $\overline z=\overline z_{N_k}$. The number of iteration steps
of Algorithm~\ref{algo:sdp_solver} is asymptotically bounded by
\begin{align*}
\limsup_{k\to \infty} N_k\leq 1+\max\left\{ 50 \ln \frac{K}{(\gamma(1-\theta))^{\rho-\tilde n}}, 0\right\},
\end{align*}
where 
\begin{align*}
K&:=\frac{(1+K_1)^\rho}{\tilde n^{\tilde n}K_2K_3},\\
K_1&:=\left( \frac{\trace \hat Y}{y_\sub{\min}}  + m \right) \norm{y^+-y^-}_\infty + \trace (A_{y^+}-C),\\
K_2&:=\frac{\det (BB^T)  }{\norm{A_{y^+}}_2^{2l}} \prod_{i=1}^m (1+\trace(A_i)),\\
K_3&:= (\lambda_\text{min} (A_{y^\circ}-C))^l (\lambda_\text{min}(A_{y^-}))^{n-l} 2^{-2m} \delta^{2m}.
\end{align*}
\end{lemma}
\begin{proof}
For all sufficiently large outer iteration indices $k\in \N_0$, Theorem~\ref{thm:Frank_Wolfe_Convergence} yields that Algorithm~\ref{algo:sdp_solver} will not return failure. Lemma~\ref{lemma:sdp_terminates} yields that it terminates with $\eta_{N_k}$ that fulfills \eqref{eq:stp_success_test} and thus \eqref{eq:stop_crit_sdp}. 

To obtain a bound on the required number of iteration steps $N_k$, we now estimate the terms on the right hand side of \eqref{eq:dual_scaling_conv_speed_bound}. 
\begin{enumerate}[(a)]
\item From our definition of the starting values in \eqref{eq:dual_scaling_shift}--\eqref{eq:dual_scaling_start2}, and Lemma~\ref{lemma:sdp_feasible}, we obtain 
\begin{align*}
\overline z_0-b^T\eta_0 &= \langle \tilde X_0, \tilde C- \tilde A_{\eta_0} \rangle\\
&=\langle \tilde X_{y_k}, \tilde C- \tilde A_{\eta_0} \rangle +\epsilon_k \langle  I, A_{\eta_0}-C  \rangle\\
& \quad {} + \langle \epsilon_k \operatorname{diag}(\1+{\mathcal A}(I)),\operatorname{diag}(y^+ - \eta_0)\rangle + \langle \epsilon_k I,\operatorname{diag}(\eta_0 - y^-)\rangle\\
%
%
%
&=\langle \tilde X_{y_k}, \tilde C- \tilde A_{y_k} \rangle + \epsilon_k K_{y_k,\eta_0 },
\end{align*}
with 
\begin{align*}
K_{y_k,\eta_0}&:=\langle \tilde X_{y_k}, \tilde A_{y_k-y^\circ} \rangle + \trace (A_{\eta_0}-C)+  \1^T(y^+-y^-)+ \mathcal A(I)^T (y^+- \eta_0)\\
&= \langle \tilde X_{y_k}, \tilde A_{y_k-y^\circ} \rangle + \trace (A_{y^+}-C)+  \1^T(y^+-y^-).
\end{align*}
By Lemma~\ref{lemma:IP_vs_sdp}, $\langle X_{y_k}, A_{y_k} \rangle=\trace(F(y_k))\leq \trace \hat Y$, so that we obtain
\begin{align*}
 \langle \tilde X_{y_k}, \tilde A_{y_k-y^\circ} \rangle &= \langle X_{y_k}, A_{y^\circ-y_k} \rangle
 \leq \frac{\norm{y^\circ-y_k}_\infty}{y_\sub{\min}} \langle X_{y_k}, A_{y_\sub{\min}\1} \rangle\\
&\leq
\frac{\norm{y^\circ-y_k}_\infty}{y_\sub{\min}} \langle X_{y_k}, A_{y_k} \rangle
 \leq \frac{\norm{y^+-y^-}_\infty}{y_\sub{\min}} \trace \hat Y.
\end{align*}

Moreover, again using Lemma~\ref{lemma:IP_vs_sdp},
\[
\langle \tilde X_{y_k}, \tilde C- \tilde A_{y_k} \rangle
\leq \trace(\hat Y-F(y_k))=r(y_k),
\]
so that we arrive at
\begin{equation*}
\overline z_0-b^T\eta_0\leq r(y_k)+ \epsilon_k K_1.
\end{equation*}
\item To estimate $\det \tilde X_0$, note that $X_{y_k}=A_{y_k}^{-1} B^T B A_{y_k}^{-1}$ has exactly $l$ non-zero eigenvalues, namely the 
eigenvalues of $B A_{y_k}^{-2} B^T\succ 0$. Also, for all $v\in \R^l$,
\begin{align*}
v^T B A_{y_k}^{-2} B^T v= \norm{A_{y_k}^{-1} B^T v}_2^2\geq \frac{1}{\norm{A_{y^+}}_2^2} \norm{B^T v}_2^2
 = \frac{1}{\norm{A_{y^+}}_2^2} v^T B B^T v
\end{align*}
which gives 
$
B A_{y_k}^{-2} B^T   \succeq \frac{B B^T}{\norm{A_{y^+}}_2^2},
$
so that we obtain
\begin{align*}
\det \tilde X_0&=\det (X_{y_k}+\epsilon_k I) \epsilon_k^{2m} \prod_{i=1}^m (1+\trace(A_i))
\geq \epsilon_k^{\tilde n-l} \det (B A_{y_k}^{-2} B^T) \prod_{i=1}^m (1+\trace(A_i))\\
&\geq \epsilon_k^{\tilde n-l} K_2.
\end{align*}
\item Now we turn to the third term $\det\tilde S_{\eta_0}$. We have that
$\operatorname{rank} C=l$ implies that at least $n-l$ eigenvalues of $A_{\eta_0}-C$ are bounded from below by $\lambda_\text{min}(A_{\eta_0})$ and thus by $\lambda_\text{min}(A_{y^-})$.
Moreover, $A_{\eta_0}-C\succeq \epsilon_k (A_{y^\circ}-C)$ gives that the other eigenvalues are bounded from below by
$\epsilon_k \lambda_\text{min} (A_{y^\circ}-C)$. This shows that
\[
\det(A_{\eta_0}-C)\geq \epsilon_k^l (\lambda_\text{min} (A_{y^\circ}-C))^l (\lambda_\text{min}(A_{y^-}))^{n-l}.
\]
Also, the convergence of $y_k\to \hat y$ and $\epsilon_k =\min\{ r(y_k),\frac{1}{2} \}\to 0$ implies that $\eta_0$ gets arbitrarily close to $\hat y$ over the course of Frank-Wolfe iterations.
Hence, for sufficiently large $k$, we have that
\[
y^+-\eta_0>\frac{1}{2}(y^+-\hat y)\geq \frac{\delta}{2}\1,\quad \text{ and } \quad \eta_0-y^->\frac{1}{2}(\hat y- y^-)\geq \frac{\delta}{2}\1.
\]
It follows that
\begin{align*}
\det (\tilde S_{\eta_0})&=\det (\tilde C- \tilde A_{\eta_0})
\geq \epsilon_k^l K_3.
\end{align*}
\end{enumerate}
Inserting (a)--(c) in Lemma~\ref{lemma:dual_scaling_convergence_speed} yields
\begin{align*}
 \overline z_j - b^T \eta_j 
&\leq  (r(y_k)+K_1\epsilon_k)^\frac{\rho}{\rho-\tilde n}  \epsilon_k^{-\frac{\tilde n}{\rho-\tilde n}} (\tilde n^{\tilde n} K_2 K_3)^{-\frac{1}{\rho-\tilde n}} \exp\left({-\frac{1}{50(\rho-\tilde n)}j}\right),
\end{align*}
Hence, for sufficiently large $k\in \N_0$, $\epsilon_k=r(y_k)$, so that
\begin{align*}
 \overline z_j - b^T \eta_j 
&\leq  K^{\frac{1}{\rho-\tilde n}} r(y_k) e^{-\frac{1}{50(\rho-\tilde n)}j}.
\end{align*}
Using that $\overline z_j$ is an upper bound for the dual objective, Theorem~\ref{thm:Frank_Wolfe_Convergence} yields, for all sufficiently large $k\in \N_0$,
\begin{align*}
\overline z_j-b^T y_k&\geq \max_{w\in\Adm}b^Tw-b^Ty_k \geq \gamma r(y_k).
\end{align*}
Combining this with the preceding estimate gives
\begin{align*}
 \overline z_j - b^T \eta_j 
&\leq \frac{K^{\frac{1}{\rho-\tilde n}}}{\gamma} e^{-\frac{1}{50(\rho-\tilde n)}j}
(\overline z_j-b^T y_k).
\end{align*}

Thus, for sufficiently large $k\in\N$, the termination criterion $\overline z_j - b^T \eta_j \leq (1-\theta)(\overline z_j-b^T y_k)$ is fulfilled if
\begin{align*}
 j
\geq 50 \ln \frac{K}{(\gamma(1-\theta))^{\rho-\tilde n}}. 
\end{align*}
This proves that, for all sufficiently large $k\in \N$, the number of inner iteration steps is bounded by
\[
N_k\leq 1+\max\left\{ 50 \ln \frac{K}{(\gamma(1-\theta))^{\rho-\tilde n}}, 0\right\}.
\]
\end{proof}

\subsection{The global search}

For the global search routine, we require a method that is guaranteed to find, after finitely many steps, a point $y_{k+1}\in \Adm$ that fulfills
\begin{equation}\label{eq:b_and_b_goal}
r(y_{k+1})\leq \kappa r(y_k).
\end{equation}

We utilize a monotonicity- and SDP-enhanced branch-and-bound method. Starting with the componentwise half-open box
\[
[y^- + \tfrac{\delta}{2} \1,y^+ - \tfrac{\delta}{2}\1),
\]
as the only entry of a LIFO stack, we repeatedly remove the top box and subdivide it by splitting an edge of maximal relative length, at the geometric mean of its endpoints. In the case of a tie, an edge is chosen with largest corresponding component of $-\nabla \trace F(\eta^+)\geq 0$. The two half-open child boxes form a disjoint partition of their parent box.
Hence, exactly one of the child boxes contains $\hat y$ whenever the parent box does.

Given a child box $Q=[\eta^-,\eta^+)$, we first check both corners. If one of them is strictly feasible and fulfills \eqref{eq:b_and_b_goal}, then we can terminate and return that corner as $y_{k+1}$. Note that we only use strict feasibility tests for our method, as this can be checked easily via a Cholesky factorization, without requiring a tolerance parameter for distinguishing a zero eigenvalue from a small negative one. If neither corner test succeeds, we discard $Q$ if
$\trace F(\eta^-)< \trace\hat Y$, or if $\eta^+$ is not strictly feasible, as either condition implies, by monotonicity, that $\hat y\not \in [\eta^-,\eta^+)$. 
Thus, all remaining boxes fulfill 
\begin{equation}\label{eq:BB_remaining}
F(\eta^+)\prec \hat Y,\quad r(\eta^+)>\kappa r(y_k),\quad \text{ and } \quad F(\eta^-)\not\prec \hat Y.
\end{equation}

We derive a further SDP-based test for discarding boxes, or finding $y_{k+1}$, from Theorem~\ref{thm:FW_with_monotonicity}. 
In the following, for a box $Q=[\eta^-,\eta^+)$, we denote the maximal relative side length with
\[
q_Q:=\max_{i=1,\ldots,m} \frac{\eta^+_i}{\eta^-_i}.
\]
Let $Q=[\eta^-,\eta^+)$ remain after the above monotonicity tests. Then we set 
\[
b:=\nabla \trace F(\eta^+), \quad Q_\delta:=[\eta^-,\eta^+ +\tfrac{\delta}{2}\1],
\]
and solve the SDP
\[
b^T y \to \text{max!}\quad \text{ s.t. } \quad y\in \Adm\cap Q_\delta,
\]
with the dual scaling algorithm from subsection \ref{subsect:sdp_solver} (with $y^-$ and $y^+$ replaced by the local bounds $\eta^-$ and
$\eta^+ +\frac{\delta}{2}\1$, respectively). 
By \eqref{eq:BB_remaining}, $\eta^+$ is strictly feasible for the 
 matrix constraint and in the interior of $Q_\delta$, so that the
iteration can be started with 
\[
\eta_0:=\eta^+, \quad \text{ and } \quad \overline z_0:=-\trace (F(\eta^+)\hat Y^{-1} F(\eta^+)).
\]
The latter is an upper bound for the dual objective by Lemma~\ref{lemma:IP_vs_sdp}.

Let $\eta_j\in \Adm\cap Q_\delta$ and $\overline z_j\in \R$ denote the dual iterates, and the dual-objective upper bounds, obtained after $j$ steps of the dual scaling iteration. If $\hat y\in Q$, then $\eta^+> \hat y$, and Theorem \ref{thm:FW_with_monotonicity} gives that, for all $j\in \N_0$,
\begin{equation}\label{eq:box_sdp_failure_crit}
\overline z_j-b^T \eta^+\geq \frac{1}{q_Q} r(\eta^+).
\end{equation}
Moreover, as in subsection \ref{subsect:sdp_solver}, if for some $\theta_Q\in (0,1)$,
\begin{equation*}
\overline z_{j+1}-b^T \eta_{j+1}\leq (1-\theta_Q) (\overline z_{j+1}-b^T \eta^+),
\end{equation*}
then $\eta_{j+1}$ fulfills
\begin{equation}\label{eq:box_sdp_success_crit}
b^T (\eta_{j+1}-\eta^+)\geq \theta_Q (\overline z_{j+1}-b^T \eta^+).
\end{equation}
Hence, by Theorem \ref{thm:FW_with_monotonicity}, if $\overline z_{j+1}$ fulfills \eqref{eq:box_sdp_failure_crit} and $\eta_{j+1}$ satisfies \eqref{eq:box_sdp_success_crit}
for some $\theta_Q\in (0,1)$, then
\begin{equation}
r(\eta_{j+1})\leq \left( 1-\frac{\theta_Q}{q_Q}\right) r(\eta^+).
\end{equation}
Consequently, if a box $Q$ is small enough to fulfill
\begin{equation}\label{eq:BB_smallness}
\theta_Q:=\frac{q_Q}{r(\eta^+)} (r(\eta^+)-\kappa r(y_k))<1,
\end{equation}
then, with the same argument as in Lemma~\ref{lemma:sdp_terminates}, we can terminate, after finitely many iterations, either with
$\overline z_{j+1}$ violating \eqref{eq:box_sdp_failure_crit}, or with $\eta_{j+1}$ fulfilling \eqref{eq:box_sdp_success_crit}.
In the first case, we can discard the box, as it cannot contain $\hat y$. In the second case, we can terminate the global search as
$y_{k+1}:=\eta_{j+1}$ fulfills \eqref{eq:b_and_b_goal}.

We summarize the resulting monotonicity- and SDP-enhanced branch-and-bound method in Algorithm \ref{algo:global_solver}. Our choice of a LIFO stack effectively limits the number of boxes stored simultaneously, and our splitting rule is motivated by reducing 
the maximal relative side length $q_Q$ of the boxes. The tie-rule based on the components of $-\nabla\trace F(\eta^+)$ is a heuristic choice that prioritizes components with the largest local influence on $\trace F$ in order to discard boxes as early as possible. 
Let us also remark that for each child box, one of the corners agrees with a corner of the parent box, so the $F$-evaluations should be cached, and only the new corner has to undergo the monotonicity checks as explained above. We omit this in the formulation of Algorithm \ref{algo:global_solver} for the sake of brevity.

\begin{algorithm}
\caption{Monotonicity- and SDP-enhanced branch-and-bound algorithm}
\label{algo:global_solver}
\begin{algorithmic}
\State \textbf{Input:} $y_k\in\Adm\setminus\{\hat y\}$, $\kappa>0$
\State \textbf{if} $r(y^+ -\frac{\delta}{2}\1)\leq \kappa r(y_k)$, \textbf{then} \textbf{return} $y_{k+1}:=y^+ -\frac{\delta}{2}\1$; \textbf{end if}
\State \textbf{initialize} a LIFO stack with $[y^- +\frac{\delta}{2}\1,y^+ -\frac{\delta}{2}\1)$ as its only entry
\Repeat{%
\State \textbf{remove} the top box $Q:=[\eta^-,\eta^+)$ from the stack
\State \textbf{subdivide} $Q$ by splitting an edge with relative length $q_Q$ at its geometric mean
\State \hfill (in the case of a tie, choose edge with larger component of $-\nabla \trace F(\eta^+)\geq 0$)

\Statex \vspace{-2ex}
\For{each child box $Q=[\eta^-,\eta^+)$}
\Statex \vspace{-2ex}

\State \textit{\% Monotonicity tests:}

\State \textbf{if} $F(\eta^-)\prec\hat Y$ and $r(\eta^-)\leq\kappa r(y_k)$, \textbf{then} \textbf{return} $y_{k+1}:=\eta^-$; \textbf{end if}
\State \textbf{if} $F(\eta^+)\prec \hat Y$ and $r(\eta^+)\leq\kappa r(y_k)$, \textbf{then} \textbf{return} $y_{k+1}:=\eta^+$; \textbf{end if}
\State \textbf{if} $\trace F(\eta^-)<\trace\hat Y$ or $F(\eta^+)\not\prec\hat Y$, \textbf{then} \textbf{discard} $Q$ and \textbf{continue}; \textbf{end if}

\Statex \vspace{-2ex}
\State \textit{\% SDP test:}
%
\If{$\theta_Q:=\frac{q_Q}{r(\eta^+)} \bigl(r(\eta^+)-\kappa r(y_k)\bigr)<1$}
    \State \textbf{set} $b:=\nabla\trace F(\eta^+)$, and $Q_\delta:=[\eta^-,\eta^+ +\frac{\delta}{2}\1]$
    \State \textbf{set} $\eta_0:=\eta^+$, and $\overline z_0:=-\trace (F(\eta^+)\hat Y^{-1} F(\eta^+))$
    \Statex \vspace{-2ex} 
    \For{$j=0,1,2,\ldots$}
     \State \textbf{set} $\eta_{j+1}$ and $\overline z_{j+1}$ by dual scaling \eqref{eq:SDP_dualscaling1}--\eqref{eq:SDP_update_eta} applied to $\Adm\cap Q_\delta$
     \If{$\overline z_{j+1}$ violates \eqref{eq:box_sdp_failure_crit}}
        \State \textbf{discard} $Q$ and \textbf{break}
     \ElsIf{$\eta_{j+1}$ fulfills \eqref{eq:box_sdp_success_crit}}
         \State \Return $y_{k+1}:=\eta_{j+1}$
     \EndIf
    \EndFor
\Else
    \State \textbf{insert} $Q$ on top of the stack
\EndIf

\EndFor
}
\Until{$y_{k+1}$ is returned}

\Statex \vspace{-2ex}
\State \textbf{Output:} $y_{k+1}\in\Adm$ with
$r(y_{k+1})\leq\kappa r(y_k)$
\end{algorithmic}
\end{algorithm}

\begin{lemma}
Let $y_k\in \Adm$, $y_k\neq \hat y$, and $\kappa>0$. Algorithm~\ref{algo:global_solver} terminates after finitely many steps
and returns $y_{k+1}\in \Adm$ that fulfills $r(y_{k+1})\leq \kappa r(y_k)$. 

If the Algorithm~\ref{algo:global_solver} does not terminate during its first pass through the repeat loop, then the number of boxes on the stack maintained by Algorithm~\ref{algo:global_solver} is bounded by
\[
m+m\left \lfloor \log_2  \frac{\ln q_0}{\ln q_\text{min}}\right\rfloor,
\] 
where
\[
q_0:=\max_{i=1,\ldots,m} \frac{y^+_i-\frac{\delta}{2}}{y^-_i+\frac{\delta}{2}}> 1,
\quad \text{ and } \quad
q_\text{min}:=1+\frac{\kappa r(y_k)}{r(y^+-\frac{\delta}{2}\1)-\kappa r(y_k)}> 1.
\]
\end{lemma}
\begin{proof}
Since $y_k\in\Adm\setminus\{\hat y\}$, we have $r(y_k)>0$. If $r(y^+-\frac{\delta}{2}\1)\leq \kappa r(y_k)$ then Algorithm~\ref{algo:global_solver} immediately terminates. Otherwise, $q_{\min}>1$ is well-defined, and, by construction, the stack will always contain a box containing $\hat y$ as this box will never be discarded.
Also, by construction and the above use of Theorem \ref{thm:FW_with_monotonicity}, the algorithm only terminates after finding $y_{k+1}\in \Adm$ that fulfills \eqref{eq:b_and_b_goal}.

After the first step of the algorithm, every box on the stack fulfills 
\begin{equation}\label{eq:min_box_length}
q_Q\geq \frac{r(\eta^+)}{r(\eta^+)-\kappa r(y_k)}= 1+ \frac{\kappa r(y_k)}{r(\eta^+)-\kappa r(y_k)}\geq
1+ \frac{\kappa r(y_k)}{r(y^+-\frac{\delta}{2}\1)-\kappa r(y_k)}= q_\text{min},
\end{equation}
since otherwise \eqref{eq:BB_smallness} is fulfilled, and the SDP part either discards the box, or finds $y_{k+1}\in \Adm$ that fulfills \eqref{eq:b_and_b_goal}. Moreover, as we only split edges with maximal relative length, \eqref{eq:min_box_length} also shows that edges with relative length below $\sqrt{q_\text{min}}$ will not be split. Hence, there can only be finitely many box splittings, so that the algorithm terminates after finitely many steps.

To estimate the maximal size of the stack, note that, if $q_0<q_\text{min}$, then the stack is initialized with one box and the algorithm finishes in the first iteration. We now consider the case $q_0\geq q_\text{min}$. Since we use a LIFO stack, and split edges with maximal relative length, after $m$ steps the maximal relative edge length is decreased to its square root. Every step removes one box from the stack and adds up to two new boxes. Hence, for all $s\in \N$, the $(1+ms)$-th box on the stack (if there is one) can have maximal relative length of $\sqrt[2^s]{q_0}$. Therefore, the $(1+ms^*)$-th box on the stack, where
\[
s^*=\min\{ s\in \N_0:\ \sqrt[2^s]{q_0}<q_\text{min} \} = 1+\left \lfloor \log_2  \frac{\ln q_0}{\ln q_\text{min}}\right\rfloor,
\]
would have maximal relative length smaller than $q_\text{min}$. As this contradicts \eqref{eq:min_box_length}, the 
bound on the number of boxes on stack is proven.
\end{proof}

\section{Application to the discretized Calder\'on problem}

We consider the discretized Calder\'on problem that arises in Electrical Impedance Tomography (EIT) with finitely many measurements and unknowns after FEM-discretization. We first summarize how this problem can be written in the matrix form considered in this work, and then we discuss an efficient implementation and dimensionality reduction.

\subsection{EIT with continuous boundary measurements}\label{subsect:EIT_FEM}

Let $\Omega$ be a bounded Lipschitz domain in dimension larger than $1$, and let $\nu$ denote the unit outer normal on $\partial \Omega$. We write $L^\infty_+(\Omega)$ for the subset of $L^\infty$-functions 
with positive essential infima. $H_\diamond^1(\Omega)$ and $L^2_\diamond(\partial \Omega)$ denote $H^1(\Omega)$, respectively, $L^2(\partial \Omega)$ with vanishing integral mean on $\partial \Omega$.

Driving an electrical current $g\in L^2_\diamond(\partial \Omega)$ through the boundary of a conducting domain $\Omega$ gives rise to an electrical potential $u\in H^1_\diamond(\Omega)$ that solves
\begin{equation}\label{eq:EIT}
\nabla \cdot (\sigma \nabla u)=0 \quad \text{ in } \Omega, \quad \text{ and } \quad \sigma\partial_\nu u|_{\partial \Omega}=g \quad \text{ on } \partial\Omega,
\end{equation}
or, equivalently, 
\begin{equation}\label{eq:var_form_EIT}
\int_\Omega \sigma \nabla u \cdot \nabla v \dx=\int_{\partial \Omega} g v \dx[s] \quad \text{ for all } v\in H^1_\diamond(\Omega).
\end{equation}

We aim to reconstruct the conductivity $\sigma\in L^\infty_+(\Omega)$ from measuring $u^{(g_k)}|_{\partial \Omega}$ for finitely many, linearly independent boundary currents $g_1,\ldots,g_l\in L^2_\diamond(\partial \Omega)$.
We assume that 
\[
\sigma(x)=\sum_{i=1}^m y_i \chi_{\mathcal P_i}(x)
\]
is piecewise constant with respect to a given pixel partition $\Omega=\bigcup_{i=1}^m \mathcal P_i$, and that we can measure the projection of the corresponding solutions $u_\sigma^{(g_k)}|_{\partial \Omega}$ to the $l$-dimensional subspace of $L^2_\diamond(\partial \Omega)$ that is spanned by $g_1,\ldots,g_l$. This is a finite-dimensional Galerkin projection of the infinite-dimensional Neu\-mann-to-Dirichlet (NtD) operator for the PDE \eqref{eq:EIT}.
Using \eqref{eq:var_form_EIT}, it follows that the measurements form a matrix $\mathcal Y\in \sS^l_{++}$ with entries
\[
\mathcal Y_{jk}:=\int_{\partial \Omega}  g_j u_\sigma^{g_k}|_{\partial \Omega}\dx[s]
=\int_\Omega \sigma \nabla u_\sigma^{g_j}\cdot \nabla u_\sigma^{g_k} \dx, \quad j,k=1,\ldots,l.
\]
The reconstruction of finitely many unknowns from finitely many measurements in this setting 
is an inverse problem for the finite-dimensional matrix-valued mapping 
\[
\mathcal F:\ \R^m_{++}\to \sS^l,\quad y\mapsto \mathcal Y,
\]
where the evaluation of $\mathcal F$ requires solving infinite-dimensional problems, i.e., the PDE \eqref{eq:EIT}.

Solving the PDE by a finite element method (FEM) corresponds to 
restricting the variational formulation \eqref{eq:var_form_EIT} to a finite-dimensional subspace $V_\sub{FEM}\subset H_\diamond^1(\Omega)$ with basis $\varphi_1,\ldots, \varphi_n\in V_\sub{FEM}$. 
The PDE solution $u^{(g_k)}\in H_\diamond^1(\Omega)$ that solves \eqref{eq:var_form_EIT} for all $v\in H_\diamond^1(\Omega)$, and $\sigma=\sum_{i=1}^m y_i \chi_{\mathcal P_i}(x)$  is thus approximated by
\[
\tilde u^{(g_k)}:=\sum_{j_1=1}^n \tilde u_{j_1}^{(g_k)} \varphi_{j_1}(x)\in V_\sub{FEM},
\]
which solves \eqref{eq:var_form_EIT} for all $\varphi_1,\ldots,\varphi_n$, i.e.
\[
\sum_{i=1}^m y_i   \sum_{j_1=1}^n \tilde u_{j_1}^{(g_k)} \int_{\mathcal P_i} \nabla\varphi_{j_1}(x) \cdot \nabla \varphi_{j_2} \dx= \int_{\partial \Omega}  g_k \varphi_{j_2}|_{\partial \Omega}\dx[s] \quad \text{ for all } j_2=1,\ldots,n.
\]
We collect the values 
\[
\int_{\mathcal P_i} \nabla\varphi_{j_1} \cdot \nabla \varphi_{j_2} \dx,\quad \text{ and } \quad
\int_{\partial \Omega}  g_k \varphi_{j_2}|_{\partial \Omega}\dx[s],
\]
for $j_1,j_2=1,\ldots,n$, respectively, for $k=1,\ldots,l$ and $j_2=1,\ldots,n$ into 
the pixel-wise stiffness matrix $A_i\in \R^{n\times n}$ and the boundary matrix $B\in \R^{l\times n}$.

Identifying $\tilde u^{(g_k)}\in V_\text{FEM}$ with its vector of expansion coefficients in $\R^n$, this yields that
\[
\sum_{i=1}^m y_i A_i \tilde u^{(g_k)} = B^T e_k,
\]
and the $(j,k)$-th EIT measurement is thus approximated by
\[
\int_{\partial \Omega} g_j \tilde u^{(g_k)}|_{\partial \Omega} \dx[s]=e_j^T B A_y^{-1} B^T e_k,\quad \text{ with } \quad
A_y=\sum_{i=1}^m y_i A_i.
\]
This shows that the matrix $Y\in \sS^{l}$ of FEM-approximated measurements for the pixel conductivity values $y\in \R^m$ is given by the finite-dimensional mapping
\[
F:\ \R^m_{++}\to \sS^l, \quad F(y)=B A_y^{-1} B^T.
\]
Note that the coercivity of the variational formulation implies that $A_\1\succ 0$, and $B$ is surjective if the FEM mesh is fine enough to resolve the boundary currents.

\subsection{Efficient numerical implementation}\label{subsect:implementation}

We summarize some remarks on efficiently implementing the algorithm that go beyond standard good practice, such as caching intermediate results, exploiting sparsity, and half-vectorization of symmetric matrices.

\paragraph{Static condensation by skeleton reduction}
In a standard FEM implementation, the pixel-interior degrees of freedom can be eliminated by static condensation, leaving only interface and exterior-boundary degrees of freedom, see, e.g.\ \cite[Sec.~8.4.3]{Ern04}. A particularly simple way to exploit this property can be formulated as a skeleton decomposition. For $i=1,\ldots,m$, we define index sets $J_i\subseteq \{1,\ldots,n\}$ by
\begin{equation*}
J_i:=\{ k\in \{1,\ldots,n\}:\ Be_k=0 \text{ and } A_j e_k=0 \ \text{ for all } j\neq i\},
\end{equation*}
and set $J_\text{skel}:=\{1,\ldots, n\}\setminus \bigcup_{i=1}^m J_i$. Then, with
\begin{align*}
A_i^\text{skel} &=(A_i)_{(J_\text{skel},J_\text{skel})} -(A_i)_{(J_\text{skel},J_i)} (A_i)_{(J_i,J_i)}^{-1}
 (A_i)_{(J_i,J_\text{skel})}, \label{eq:direct_skeleton_A}
\end{align*}
and $B^\text{skel}=B_{(\cdot,J_\text{skel})}$, we have that, for all $y\in \R^m_{++}$, 
\[
B A_y^{-1} B^T = B^\text{skel} \left(A_y^\text{skel} \right)^{-1} (B^\text{skel})^T.
\]
Also, $A_i^\text{skel}$ and $B^\text{skel}$ still fulfill our assumptions as introduced in subsection \ref{subsect:setting}. 

We give an elementary proof for this in Lemma~\ref{lemma:skeleton}. 
Let us stress again that this is essentially static condensation as $J_i$ contains the indices of pixel-interior nodes. But Lemma~\ref{lemma:skeleton} holds independently of this interpretation as it
is based solely on the structure of the matrices $A_i$, and $B$. This allows us to reduce the dimension $n$ without any knowledge of the FEM implementation details. In our implementation, we use this as a first preprocessing step.

\kommentar{
\paragraph{Low rank factorizations}

The pixel-wise stiffness matrices $A_i$ are usually of low rank. After the dimension reduction by skeleton decomposition (we do not change notation), we therefore apply another preprocessing step, where we calculate $R_i\in \R^{r_i\times n}$ with
\[
A_i=R_i^T R_i.
\]
Then, in the dual scaling iteration, the calculation of ?? can be replaced by ??. Also, the calculation of ?? can be replaced by ??.
}

\paragraph{Dual scaling step-size}
The choice $\alpha=0.4$ in \eqref{eq:SDP_update_eta} guarantees strict
feasibility and a decrease of the Tanabe-Todd-Ye potential function by
at least $\frac{1}{50}$, cf.\ Lemma~\ref{lemma:dual_scaling_convergence_speed}(a), and the underlying result in \cite[Thm.~2]{benson2000solving}.
Our convergence result and the bound on the iteration number stay valid for larger step sizes $\alpha>0.4$ as long as they still fulfill these two properties.
Hence, we can proceed as in \cite{benson2008algorithm,toh2002steplengths}. We calculate the maximal $\alpha$ for which feasibility holds. Starting with $0.95$ of this value, we then geometrically decrease it until we either find $\alpha>0.4$ for which the potential decrease is larger than $\frac{1}{50}$, or we fall back to $\alpha=0.4$. To calculate the maximal feasible $\alpha$ note that, for a strictly dual feasible $y$ and a direction $d\in \R^m$, and step size $s>0$,
\[
A_{y+sd}\succ C \quad \text{ if and only if } 
s\lambda<1,
\]
where $\lambda:=\lambda_{\mathrm{max}}(-A_d,A_y-C)$ denotes the largest generalized eigenvalue, that
is, $-A_dv=\lambda(A_y-C)v$ for some $v\neq0$. 
It can be estimated using MATLAB's \texttt{eigs} function (which uses the Lanczos iteration recommended in \cite{toh2002steplengths}). The decrease in the Tanabe--Todd--Ye potential function
can be efficiently calculated via Lemma~\ref{lemma:potential_decrease}.

\kommentar{
\paragraph{Combination with other solvers}
Let us also emphasize that our global solver can accept and proceed with whatever $y_{k+1}$ as long as it is admissible and fulfills the globalization condition. This is typically computationally cheap to check as it only requires an inversion of $A_{y_{k+1}}$ which is typically needed anyway in the iteration. Hence, all subroutine iterates including dual scaling iterates, additional local solver, or global solver iterates should always be checked for this property to allow early termination.

Moreover, regarding the choice of algorithm design parameters, let us note that it seems natural to update $b$ as often as possible, as this does not come with any additional cost for the dual scaling method. Also, a fallback to the computationally expensive global solver should be avoided when possible. This motivates to choose $\theta>0$ close to zero, and $\kappa$ close to $1$.}

\begin{appendices}

\section{Efficient numerical implementation}
\label{app:implementation}

Here we collect some useful details for an efficient numerical implementation as summarized in subsection \ref{subsect:implementation}.

\begin{lemma}\label{lemma:skeleton}
Let $0\neq A_i\in \sS^n_+$, $i=1,\ldots,m$, $A_\1\succ 0$, and let $B\in \R^{l\times n}$ have full row rank.
For all $i=1,\ldots,m$, we define an index set $J_i\subseteq \{1,\ldots,n\}$ by
\begin{equation}\label{eq:skel_Ji}
J_i:=\{ k\in \{1,\ldots,n\}:\ Be_k=0 \text{ and } A_j e_k=0 \ \text{ for all } j\neq i\},
\end{equation}
and set $J_\text{skel}:=\{1,\ldots, n\}\setminus \bigcup_{i=1}^m J_i$. Then 
\begin{enumerate}[(a)]
\item The sets $J_\text{skel}$, $J_1$, \ldots, $J_m$ are pairwise disjoint and form a partition of $\{1,\ldots,n\}$.
\item $(A_i)_{(J_i,J_i)}$ is invertible for all $i=1,\ldots,m$.
\item Every $u=A_y^{-1} B^T g$ with $y\in \R^m_{++}$ and $g\in \R^l$ satisfies, for all $i=1,\ldots,m$,
\[
u_{(J_i)}=-(A_i)_{(J_i,J_i)}^{-1}   (A_i)_{(J_i,J_\text{skel})} u_{(J_\text{skel})}.
\]
\item With $n_\text{skel}=|J_\text{skel}|$, define $P\in \R^{n\times n_\text{skel}}$ by
\[
P_{(J_\text{skel},:)}=I,\quad \text{ and } \quad P_{J_i,:}=-(A_i)_{(J_i,J_i)}^{-1}   (A_i)_{(J_i,J_\text{skel})}.
\]
Then, with
\[
A_i^\text{skel}=P^T A_i P\in \sS^{n_\text{skel}}, \quad \text{ and } \quad B^\text{skel}=BP\in \R^{l\times n_\text{skel}}
\]
we have that $A_i^\text{skel}\succeq 0$, and $B^\text{skel}$ has full row rank. Moreover, for all $y\in \R^m_{++}$, and $A_y^\text{skel}:=\sum_{i=1}^m y_i A_i^\text{skel}$, we have that $A_y^\text{skel}\succ 0$, and that
\[
B A_y^{-1} B^T = B^\text{skel} \left(A_y^\text{skel} \right)^{-1} (B^\text{skel})^T.
\]
\item $A_i^\text{skel}$ and $B^\text{skel}$ fulfill 
\begin{align}
A_i^\text{skel}
&=(A_i)_{(J_\text{skel},J_\text{skel})}
 -(A_i)_{(J_\text{skel},J_i)}
 (A_i)_{(J_i,J_i)}^{-1}
 (A_i)_{(J_i,J_\text{skel})}, \label{eq:direct_skeleton_A}\\
B^\text{skel}&=B_{(\cdot,J_\text{skel})}. \label{eq:direct_skeleton_B}
\end{align}
\end{enumerate}
\end{lemma}
\begin{proof}
If $k\in J_i\cap J_j$ then $A_je_k=0$ for all $j=1,\ldots,m$ with $j\neq i$, so that $A_\1 e_k=0$ which contradicts $A_\1\succ 0$. Also, $A_\1\succ 0$ implies that
$(A_i)_{(J_i,J_i)}=(A_\1)_{(J_i,J_i)}\succ 0$, so that (a) and (b) are proven.

For all $u\in \R^n$, and $i=1,\ldots,m$, we have
\[
A_i u = (A_i)_{(\cdot,J_\text{skel})} u_{(J_\text{skel})} + (A_i)_{(\cdot,J_i)} u_{(J_i)}.
\]
For $u=A_y^{-1} B^T g$, with $y\in \R^m_{++}$ and $g\in \R^l$, we have, for all $k\in J_i$, 
\[
0=e_k^T B^T g= e_k^T A_y u = y_i e_k^T  A_i u = y_i e_k^T (A_i)_{(\cdot,J_\text{skel})} u_{(J_\text{skel})} + y_i e_k^T (A_i)_{(\cdot,J_i)} u_{(J_i)}
\]
This shows that every $u=A_y^{-1} B^T g$ with $y\in \R^m_{++}$ and $g\in \R^l$ fulfills
\[
0=(A_i)_{(J_i,J_\text{skel})} u_{(J_\text{skel})} + (A_i)_{(J_i,J_i)} u_{(J_i)},
\]
so that (c) is proven. 

Moreover, for all $u=A_y^{-1} B^T g$ with $y\in \R^m_{++}$ and $g\in \R^l$, $u=Pu_{(J_\text{skel})}$, and thus
\[
P^T A_y P u_{(J_\text{skel})}= P^T B^T g
\]
By definition, $\operatorname{rank}P=n_\text{skel}$, so that $P$ is injective, which gives
$A_y^\text{skel}\succ 0$. Hence,
\[
u_{(J_\text{skel})}= (P^T A_y P)^{-1} P^T B^T g.
\]
Hence, for all $y\in \R^m_{++}$, $g\in \R^l$, and $u:=A_y^{-1} B^T g$
\[
BA_{y}^{-1}B^T g=  B P u_{(J_\text{skel})} = BP (P^T A_y P)^{-1} P^T B^T g
\]
This also implies that $B^\text{skel}=BP$ has full row rank.

(e) follows from (d) using the definition of $P$ and the fact that our definition \eqref{eq:skel_Ji} implies that $B_{(\cdot,J_j)}=0$ for all $j=1,\ldots m$, and that for $j\neq i$ we also have
$(A_i)|_{(J_j,:)}=0$ and $(A_i)|_{(:,J_j)}=0$.
\kommentar{ 
To prove (e), note that, by definition \eqref{eq:skel_Ji}, $(A_i)|_{:,J_j}=0$ and $(A_i)|_{J_j,:}=0$, for all $j\neq i$. Together with $P_{(J_\text{skel},:)}=I$, and $P_{J_i,:}=-(A_i)_{(J_i,J_i)}^{-1}   (A_i)_{(J_i,J_\text{skel})}$, this gives that 
\begin{align*}
(A_i P)_{(J_j,:)}&=0\\
(A_i P)_{(J_i,:)}&=A_i|_{(J_i,:)} P =A_i|_{(J_i,J_i)} P_{(J_i,:)}+ A_i|_{(J_i,J_\text{skel})} P_{(J_\text{skel},:)}=0,\\
(A_i P)_{(J_\text{skel},:)}&=A_i|_{(J_\text{skel},:)} P =A_i|_{(J_\text{skel},J_i)} P_{(J_i,:)}+ A_i|_{(J_\text{skel},J_\text{skel})}\\
&= -A_i|_{(J_\text{skel},J_i)} (A_i)_{(J_i,J_i)}^{-1}   (A_i)_{(J_i,J_\text{skel})}
+ A_i|_{(J_\text{skel},J_\text{skel})}
\end{align*}
so that, 
\begin{align*}
P^T A_i P&=(P)|_{(J_\text{skel},:)}^T (A_i P)|_{(J_\text{skel},:)} 
+(P)|_{(J_i,:)}^T (A_i P)|_{(J_i,:)}\\
&=(P)|_{(J_\text{skel},:)}^T (A_i P)|_{(J_\text{skel},:)}  = A_i|_{(J_\text{skel},J_\text{skel})} - 
A_i|_{(J_\text{skel},J_i)} (A_i)_{(J_i,J_i)}^{-1}   (A_i)_{(J_i,J_\text{skel})}.
\end{align*}
Likewise, by definition \eqref{eq:skel_Ji}, $B_{(\cdot,J_j)}=0$ for all $j=1,\ldots m$, so that
\[
BP=B_{(:,J_\text{skel})} P_{(J_\text{skel},:)}==B_{(:,J_\text{skel})}.
\]
}
\end{proof}

The following lemma allows us to calculate the decrease of the Tanabe-Todd-Ye potential function \eqref{eq:Tanabe}
for a dual-scaling step that updates only the dual variable.

\begin{lemma}\label{lemma:potential_decrease}
Let $\tilde X_j$ be strictly primal feasible, and $\tilde X_{j+1}=\tilde X_j$. Let $\eta_j$ and $\eta_{j+1}$ be strictly dual feasible. 
Then, with $\overline z_{j}=\langle \tilde C,\tilde X_j \rangle$,
\begin{align}\label{eq:Psi_change}
\Psi(\tilde X_{j+1},\tilde S_{\eta_{j+1}})
&= \Psi(\tilde X_j,\tilde S_{\eta_j}) +\rho\ln\left(  \frac{\overline z_j-b^T\eta_{j+1}} {\overline z_j-b^T\eta_j} \right)
-\ln\left( \frac{\det\tilde S_{\eta_{j+1}}}{\det\tilde S_{\eta_j}} \right).
\end{align}
Moreover, for all strictly dual feasible $y\in \Adm$, the determinant of the dual slack matrix satisfies
\begin{align*}
\ln\det\tilde S_y
&= \ln\det A_y + \ln\det(\hat Y-F(y)) -\ln\det\hat Y + \sum_{i=1}^m \ln(y_i^+-y_i) + \sum_{i=1}^m\ln(y_i-y_i^-).
\end{align*}
\end{lemma}
\begin{proof}
Equation \eqref{eq:Psi_change} immediately follows from the definition of the Tanabe-Todd-Ye potential function \eqref{eq:Tanabe}.
For the determinant of the dual slack matrix, we obtain from the block structure of $\tilde S_y$ that
\[
\ln\det\tilde S_y=\ln \det (A_y-C) + \sum_{i=1}^m \ln(y_i^+-y_i) + \sum_{i=1}^m\ln(y_i-y_i^-),
\]
so that the Schur complement identity from Lemma~\ref{lemma:Schur_argument} yields the assertion.
\end{proof}

\end{appendices}

\bibliography{literaturliste}

\end{document}